\documentclass[12pt]{article}
\usepackage{amsmath}
\usepackage{amssymb,amsbsy,amsmath}
\usepackage{graphicx}
\usepackage{dsfont}
\usepackage{color}
\usepackage[margin=1in]{geometry}

\usepackage{pdfsync}
\usepackage{hyperref}
\usepackage{prettyref}

\allowdisplaybreaks[1]

\newrefformat{def}{Definition \ref{#1}}
\newrefformat{app}{Appendix \ref{#1}}

\newtheorem{theorem}{Theorem}
\newtheorem{proposition}{Proposition}
\newtheorem{lemma}{Lemma}
\newtheorem{corollary}{Corollary}
\newtheorem{remark}{Remark}

\newtheorem{definition}{Definition}
\newtheorem{conjecture}{Conjecture}

\renewcommand{\phi}{\varphi}

\newcommand{\nL}{\mathcal{L}}
\newcommand{\evonep}{f_1^{+}}
\newcommand{\evone}{f_1}

\newcommand{\smlhf}{\scriptscriptstyle -1/2}

\renewcommand{\P}{\mathbb{P}}
\newcommand{\E}{\mathbb{E}}

\newcommand{\Z}{\mathbb{Z}}
\newcommand{\R}{\mathbb{R}}

\newcommand{\cA}{\mathcal{A}}

\newcommand{\cT}{\mathcal{T}}

\def\ds1{\mathds{1}}
\renewcommand{\epsilon}{\varepsilon}
\newcommand{\eps}{\epsilon}

\newcommand{\wh}{\widehat}

\renewcommand{\tilde}{\widetilde}

\newlength{\minipagewidth}
\newcommand{\Ber}{\mathop{\mathrm{Ber}}}

\newcommand{\beq}{\begin{equation}}
\newcommand{\eeq}{\end{equation}}

\newcommand{\beqa}{\begin{eqnarray}}
\newcommand{\eeqa}{\end{eqnarray}}

\newcommand{\beqan}{\begin{eqnarray*}}
\newcommand{\eeqan}{\end{eqnarray*}}

\def\ba#1\ea{\begin{align*}#1\end{align*}} %\ba = \begin{algin*}, \ea = \end{align*}
\def\banum#1\eanum{\begin{align}#1\end{align}} %\banum = \begin{algin}, \eanum

\newcommand{\PT}{\overline{\mathbb{P}}} % P bar. (P conditional on G being typical.)

\newcommand{\BlackBox}{\rule{1.5ex}{1.5ex}}  % end of proof
\newcommand{\Typ}{\mathcal{T}_{n,p}}  % Typical G(n,p)
\newcommand{\conc}{\mathcal{R}}  % Small ball concentration
\newcommand{\Tr}{\mathrm{Tr}} % Trace
\newenvironment{proof}{\par\noindent{\bf Proof\ }}{\hfill\BlackBox\\[2mm]}

\begin{document}
\title{Braess's paradox for the spectral gap in random graphs and delocalization of eigenvectors}
\author{Ronen Eldan
	\thanks{Weizmann Institute of Science; \texttt{ronen.eldan@weizmann.ac.il}.}
	\and
	Mikl\'os Z.\ R\'acz
	\thanks{University of California, Berkeley; \texttt{racz@stat.berkeley.edu}.}
	\and
	Tselil Schramm
	\thanks{University of California, Berkeley; \texttt{tschramm@cs.berkeley.edu}.}
}
\date{\today}

\maketitle

\begin{abstract}
We study how the spectral gap of the normalized Laplacian of a random graph changes when an edge is added to or removed from the graph. There are known examples of graphs where, perhaps counterintuitively, adding an edge can decrease the spectral gap, a phenomenon that is analogous to Braess's paradox in traffic networks. We show that this is often the case in random graphs in a strong sense. More precisely, we show that for typical instances of Erd\H{o}s-R\'enyi random graphs $G(n,p)$ with constant edge density $p \in (0,1)$, the addition of a random edge will decrease the spectral gap with positive probability, strictly bounded away from zero. To do this, we prove a new delocalization result for eigenvectors of the Laplacian of $G(n,p)$, which might be of independent interest.
\end{abstract}

\section{Introduction}
The spectral gap of the Laplacian of a graph is an important quantity that
relates to conductance properties of a graph. For instance, various notions of
mixing time of a random walk on a graph are intimately related to the inverse
of the spectral gap, which is known as the relaxation time (see,
e.g.,~\cite{levin2009markov}).  Generally speaking, one expects graphs with
more edges to have better conductance properties, and, accordingly, a larger
spectral gap.  However, perhaps counterintuitively, there are examples where
adding an edge to a graph decreases its spectral gap. For example, in the
barbell graph (two expanders connected by a single edge), adding an edge within
either one of the expanders will decrease the spectral gap.

This is analogous to Braess's paradox in traffic networks, which states that
the addition of a new road can increase congestion~\cite{braess1968paradoxon}.
Since its discovery in 1968, this phenomenon has been widely studied. There are
many works that give analytical conditions for when adding an edge does or does
not yield an improvement in congestion (see,
e.g.,~\cite{steinberg1983prevalence}), and recently there have been several
works studying the prevalence of this phenomenon in random
networks~\cite{valiant2010braess,chung2010braess,chung2012braess}. These works suggest that Braess's paradox is a common occurence in many
settings.

One of the two objectives of this paper is to study how the spectral gap of the normalized Laplacian of an
Erd\H{o}s-R\'enyi random graph $G(n,p)$ changes when an edge is added to the
graph.  We show that for fixed $p \in \left( 0, 1 \right)$, the addition of a
random edge will decrease the spectral gap with positive probability.  Our main
finding is thus that the counterintuitive phenomenon that is analogous to
Braess's paradox holds in a strong sense. Our proof of this result relies on
 a certain kind of
delocalization of the second eigenvector of the normalized Laplacian of $G(n,p)$.
Showing that this occurs and exploring further delocalization results is the second objective of the paper.

\subsection{A conjecture of F. Chung and our related results}\label{sec:chung}

Our paper is motivated by a question of Fan Chung; to state her conjecture and
our results precisely, we first introduce some notation.  For a graph $G =
\left( V, E \right)$, let $A \equiv A_G$ denote its adjacency matrix, and let
$D \equiv D_G$ denote the diagonal matrix with the degrees of the corresponding vertices on
the diagonal.  The (combinatorial) Laplacian of $G$ is defined as $L \equiv L_G = D - A$, and
the symmetric normalized Laplacian is defined as
$\nL\equiv \nL_G = D^{-1/2} L D^{-1/2} = I - \wh{A}$, where $\wh{A} = D^{-1/2} A D^{-1/2}$.
Let
$\lambda_1 \left( M \right) \leq \lambda_2 \left( M \right) \leq \dots \leq \lambda_n \left( M \right)$
denote the eigenvalues of an $n \times n$ symmetric matrix $M$ in increasing order.
It is easy to see that $\lambda_1 \left( L_G \right) = \lambda_1 \left( \nL_G \right) = 0$.
The value of $\lambda_2 \left( \nL_G \right)$ is called the \emph{spectral
gap of the normalized Laplacian}, or just the \emph{spectral gap} for short.

For a graph $G$, let $r_{-} \left( G \right)$ denote the fraction of edges
$e$ such that if the edge $e$ were to be removed from the graph, the spectral
gap decreases; also let $r_{+} \left( G \right) := 1 - r_{-} \left( G
\right)$, which is the fraction of edges such that if this edge were to be
removed from the graph, the spectral gap increases.  One might first
guess that for a reasonable graph $G$, $r_{-} \left( G \right)$ is close to $1$,
i.e., that the removal of a single edge will decrease the spectral gap for most
edges.  However, based on empirical evidence to the contrary, Fan Chung
conjectured that this is not the case for Erd\H{o}s-R\'enyi random graphs.

\begin{conjecture}[Chung, 2014]
 Let $p \in \left( 0, 1 \right)$ be fixed. There exists a constant $c = c\left( p \right) > 0$ such that
\[
 \P \left[ r_{+} \left( G \left( n, p \right) \right) \geq c \right] \to 1
\]
as $n \to \infty$.
\end{conjecture}
\begin{remark}
This question is only interesting for the spectral gap of the \emph{symmetric normalized Laplacian}. 
The \emph{combinatorial Laplacian} can be written as a sum over edges of the graph of appropriate positive semidefinite matrices. 
Therefore when removing an edge from a graph, the spectral gap of the combinatorial Laplacian cannot increase. 
\end{remark}

In a similar vein, it is natural to ask how the spectral gap changes when an
edge is added between two nodes that were not previously connected.  For a
graph $G$ one can define the quantities $a_{+} \left( G \right)$ and $a_{-}
\left( G \right)$ as the proportion of ``non-edges'' (pairs of points not connected by an edge) for which adding an edge in its place increases or decreases the spectral gap, respectively.
The main result
of this paper mirrors Chung's conjecture in the case of adding a single edge to
a random graph.

\begin{theorem}\label{thm:main}
 Let $p \in \left( 0, 1 \right)$ be fixed. There exists a constant $c = c\left( p \right) > 0$ such that
\[
 \P \left[ a_{-} \left( G \left( n, p \right) \right) \geq c \right] \to 1
\]
as $n \to \infty$.
Moreover, one can take $c = 1/8 - \eta$, for any constant $\eta > 0$.
\end{theorem}

As a corollary, we get the following weaker version of Chung's conjecture.

\begin{corollary}
 Let $p \in \left( 0, 1 \right)$ be fixed. There exists a constant $c = c\left( p \right) > 0$ such that
\[
\limsup_{n \to \infty} \P \left[ r_{+} \left( G \left( n, p \right) \right) \geq c \right] \geq c.
\]
\end{corollary}

\begin{remark}
While we state our results for constant $p \in \left( 0, 1 \right)$, 
our proofs show that they hold also when $p = n^{-\eps}$ for some small $\eps >0$. 
We did not try to optimize the dependence on $p$, 
and it is possible that a similar approach could show that these results hold for smaller $p$ as well. 
\end{remark}

\subsection{Delocalization of eigenvectors}

It turns out that the proof of Theorem~\ref{thm:main} can be reduced to a
question about the second eigenvector of the normalized Laplacian of $G(n,p)$.
While the eigenvalues of matrices associated with $G(n,p)$ are very well understood
(see, for example,~\cite{chung2004eigenvalues,chung2011eigenvalues} and the references therein),
little is known about the corresponding eigenvectors. 
It is widely believed that these eigenvectors are typically \emph{delocalized},
in the sense that most of their mass is concentrated on entries whose magnitudes are roughly of the same order.
There are several ways of formalizing this intuition.

Perhaps the most common approach is to bound the $\ell_q$ norms of
eigenvectors.\footnote{We use the notation $\ell_q$ instead of the classical
$\ell_p$ in order to avoid confusion with the edge probability $p$.} If $v$ is
an eigenvector of a matrix $M$ with $\left\| v \right\|_2 = 1$, then a lower bound
on its $\ell_q$ norm is given by the relation between the norms:
for $q \geq 2$,
we have
$\left\| v \right\|_q
\geq n^{-1/2 + 1 / q} \left\| v \right\|_2
= n^{-1/2 + 1 / q}$.
If it can be shown that with high probability $\left\| v \right\|_q$
is \emph{at most} $n^{-1/2 + 1/q}$ times a polylogarithmic factor of $n$, then
the eigenvector $v$ is said to be \emph{delocalized in the $\ell_q$ sense}.

Recently there have been several works that have shown delocalization of
eigenvectors of $A_{G(n,p)}$ in the sense above. 
In particular, Erd\H{o}s~et~al.~\cite{erdHos2013spectral} showed that 
there exists a constant $C$ such that 
for every fixed $p \in \left( 0, 1 \right)$
we have 
$\left\| v \right\|_\infty 
\leq
\frac{\left( \log \left( n \right) \right)^C}{\sqrt{n}}$ 
for all unit eigenvectors $v$ with high probability. 
See also~\cite{dekel2011nodal,aroraeigenvectors,tran2010eigenvectors}. 
More widely, the $\ell_q$-delocalization of eigenvectors of Wigner matrices and more
general classes of random matrices has been intensively studied in the recent
past, see, for example,~\cite{tao2010deloc,tao2011deloc,tao2012deloc,rudelson2014deloc}.

However, in order to prove Theorem~\ref{thm:main}, we require a complementary sense
of delocalization of the second eigenvector of $\nL_{G(n,p)}$.  Namely, we need
to show that a constant fraction of the entries of the eigenvector have
magnitude on approximately the same order. While some results in this vein are known for the
\emph{first} eigenvector of the adjacency matrix~\cite{mitra2009entrywise} (and follow
in a straightforward manner from bounds on the degrees for the first
eigenvector of the symmetric
normalized Laplacian), to our knowledge there are no previous results about the
remainder of the spectrum. We obtain the following result.

\begin{theorem}\label{thm:deloc}
Let $v_2$ denote the second eigenvector of $\nL_{G\left( n , p \right)}$.
For every fixed $p \in \left( 0, 1 \right)$
and fixed $\eta \in \left( 0, 1/2 \right)$
there exists a constant $C = C \left( p, \eta \right)$
such that
\[
 \P \left[ \frac{1}{n} \# \left\{ i \in \left[ n \right] : \left| v_2 \left( i \right) \right| \geq \frac{1}{\sqrt{n} \left( \log\left( n \right) \right)^C } \right\} \geq 1/2-\eta \right] \to 1,
\]
as $n \to \infty$.
\end{theorem}

\begin{remark}
Our proof of this result also shows that the same conclusion holds for any eigenvector $v$ of $\nL_{G\left(n,p\right)}$ 
for which the corresponding eigenvalue $\lambda$ satisfies the inequality 
$\left| 1 - \lambda \right| \geq \left( 1 - \lambda_2 \left( \nL_{G\left(n,p\right)} \right) \right) /  \log \left( n \right)$. 
We omit the details. 
\end{remark}

\begin{remark}
The result above also holds for the normalized adjacency matrix
$D^{-1/2} A D^{-1/2}$ of $G\left( n, p \right)$,
% $D_{G\left( n, p \right)}^{-1/2} A_{G\left( n, p \right)} D_{G\left( n, p \right)}^{-1/2}$,
since the eigenvectors of the normalized Laplacian and those of the normalized adjacency matrix are the same.
\end{remark}

\begin{remark}
Theorem~\ref{thm:deloc} does not follow from the delocalization results in the $\ell_q$
sense, since these do not rule out the possibility that the mass of the vector
is concentrated on a sublinear number of coordinates.
\end{remark}

Before proving Theorem~\ref{thm:deloc}, 
as a warm-up we first prove an analogous result for the
(unnormalized) adjacency matrix $A_G$ of $G=G(n,p)$,
for all eigenvectors $v_j(A_G)$ with $j > 1$.
This proof contains most of the main ideas of the proof for the normalized case,
but is somewhat simpler.

Shortly before the writing of this manuscript was completed, we discovered that Rudelson and Vershynin independently proved a delocalization result of the same type for matrices with independent entries~\cite{RVwip}.
Their results are stronger, in the sense that they apply to a $(1-o(1))$-fraction of the entries; however they cannot be applied to Chung's conjecture because the normalization of the Laplacian introduces dependencies between the entries.
On a related note, Nguyen, Tao and Vu very recently proved a related result on non-degeneration of eigenvectors for a certain class of matrices with independent entries~\cite{NTV15}.
\subsection{Approach}

\subsubsection{Decreasing spectral gap}

We first obtain a general sufficient condition under which the addition of an edge causes the spectral gap to decrease.
Given a graph $G$, and another graph $G^+$ obtained from $G$ by adding a single edge, consider the second eigenvector $f_2$ of the normalized Laplacian $\nL_G$.
If $f_2$ has a smaller Rayleigh quotient in $G^+$ than in $G$, then the spectral gap decreases.
This event can be explicitly expressed using $f_2$, $\lambda_2 \left( \nL_G \right)$, and the degrees of vertices in $G$, giving an explicit sufficient condition for the spectral gap to decrease in general graphs. See Lemma~\ref{lem:adding_edge} for details.

Next, we specialize this general condition to Erd\H{o}s-R\'enyi random graphs.
Simple calculations reveal that a sufficient condition for the spectral gap to decrease with constant probability is to have a constant fraction of entries of $f_2$ have the same order of magnitude.
Thus a delocalization result of the type previously described would complete the proof.

\subsubsection{Delocalization}

For our definition of delocalization, it suffices to show that a vector
with too many small entries cannot be an eigenvector of the symmetric normalized adjacency matrix $\wh{A} = \wh{A}_{G(n,p)}$
for a typical instance of $G(n,p)$.
To this end, suppose that a vector $v$ with $\Vert v \Vert_2=1$
has many of its coordinates smaller in absolute value than $\delta \ll n^{-1/2}$,
and that $v$ is also an eigenvector of $\wh{A}$.
For a typical instance of $G(n,p)$, it is known that the second largest eigenvalue
of $\wh A$ is $\Theta(n^{-1/2})$.
Intuitively, each entry of $\wh A v$ is close to a
sum of Bernoulli random variables scaled by the entries of $v$. If $v$ is an
eigenvector, then for all of its small entries, the Bernoulli sum must land in
an interval of size $\delta \times \Theta(n^{-1/2})$ and for this event to occur
simultaneously for all of the many small entries is very unlikely.
To formalize this fact we use a standard Littlewood-Offord-type estimate,
together with a result of Rudelson and Vershynin.

By taking a suitable enumeration of discrete approximations for the eigenvector
(very roughly speaking, an $\epsilon$-net in the subset of ``localized'' eigenvectors),
one can make sure that the above holds with a probability small enough
so that a union bound can be used to show that none of these approximations are in fact likely
on a typical instance of the graph.

\subsection{Further notation}

Let $\mathds{1}$ denote the $n$-dimensional vector with all entries equal to $1$ and
let $\vec{1}$ denote the $n$-dimensional unit vector whose every coordinate is equal to $1/\sqrt{n}$.
The dimension $n$ will be implicit in the context of appearance, and will not be explicitly noted.
For $S\subseteq \left[n \right]$, $\vec{1}_S$ denotes the $\left|S\right|$-dimensional unit vector whose every coordinate is equal to $1/\sqrt{\left|S\right|}$ (and whose coordinates are identified with $S$). For a vector $v \in \R^n$, by $\left\| v \right\| \equiv \left\| v \right\|_2$ we denote its standard Euclidean norm. For a vector $v \in \R^n$, and a real number $r \in \R$, we denote by $B\left( v, r \right)$ the Euclidean ball of radius $r$ around $v$.

For a graph $G = \left( V, E \right)$ and a vertex $v \in V$, let $d_v$ denote the degree of $v$ if the graph $G$ is clear from context.
Recall our notation for various matrices associated with $G$ at the beginning of Section~\ref{sec:chung};
for all such matrices, we omit the subscript when the graph $G$ is clear from context.

We denote the eigenvalues of an $n \times n$ symmetric matrix $M$ by $\left\{ \lambda_i \left( M \right) \right\}_{i = 1}^{n}$.
For Laplacian matrices we order the eigenvalues from smallest to largest (as in Section~\ref{sec:chung}),
but for adjacency matrices it is more natural to order the eigenvalues from largest to smallest:
$\lambda_1 \left( M \right) \geq \lambda_2 \left( M \right) \geq \dots \geq \lambda_n \left( M \right)$.
In the rest of the paper we therefore follow this convention.
Let $\evone^G$ denote the normalized eigenvector corresponding to $\lambda_1( \wh
A_G)$ and $\lambda_1(\nL_G)$, and recall that $\evone^G \propto D_G^{1/2}
\vec{1}$.

A sequence of events $\left\{ E_n \right\}$ is said to hold \emph{asymptotically almost surely} if $\lim_{n \to \infty} \P(E_n) = 1$,
and is said to hold \emph{with high probability} if for every constant $c>0$ there exists $n_c>0$ such that for all $n>n_c$ one has $\P(E_n) < n^{-c}$.

\section{Preliminaries: typical instances of $G(n,p)$}\label{sec:typical}
In our proof we use several properties of a ``typical'' instance of $G(n,p)$.
We first list these properties, and then show that each
one holds with high probability over $G(n,p)$.

\begin{definition} \label{def:typical}
    We say that a graph $G = (V,E)$is a \emph{typical} instance of $G(n,p)$, denoted by $G \in \Typ$, if the following properties hold:
    \begin{enumerate}
	\item \label{pr:degrees}
	    The degrees of all vertices are close to their expectation, in the following specific sense:
	    \[
		\forall v \in V, \ np - \log( n ) \cdot \sqrt{np} \le d_v \le np +
		\log( n ) \cdot \sqrt{np},
	    \]
	    and furthermore the sum of all degrees is also close to its expectation in the following specific sense:
	    \[
		n^2 p - n \log \left( n \right) \leq \sum_{v \in V} d_v \leq n^2 p + n \log \left( n \right).
	    \]
	\item \label{pr:eval}
	    The eigenvalues of the normalized and unnormalized adjacency
	    matrices are not far from their expectations, in the following specific sense:
	    \begin{align*}
		\ np - \log( n ) \cdot \sqrt{n} \le \lambda_1(A) &\le np +
		\log( n ) \cdot \sqrt{n}, \\
		\left| \lambda_i(A) \right| &\leq 3\sqrt{np(1-p)} \qquad \text{ for } \ \  2 \le i \le n,
	    \end{align*}
		and
	    \[
		\lambda_1(\wh A) = 1, \qquad
		\left| \lambda_i(\wh A) \right| \leq \frac{8}{\sqrt{np}} \qquad \text{ for } \ \  2 \le i \le n.
	    \]
	\item \label{pr:discrepancy}
	    For every subset of vertices $S \subseteq V$, we have
	    \[
	    \Bigl \vert |E(S)| - p \left (|S| \atop 2 \right ) \Bigr \vert \leq n^{3/2}
	    \]
	    where $E(S)$ denotes the edges whose endpoints both lie in $S$.

    \end{enumerate}
\end{definition}
We denote by $\PT$ the distribution of an Erd\H{o}s-R\'enyi random graph $G(n,p)$ conditioned on it being typical, i.e.,
\[
  \PT \left( \cdot \right) := \P \left( \cdot \, \middle| \, G \in \Typ \right).
\]
The following result is well known.
\begin{theorem}\label{thm:typical}
    For every fixed $p \in \left ( 0, 1 \right )$ we have $\P(G(n,p) \in \Typ) \to 1$ as $n \to \infty$.
\end{theorem}
\begin{proof}
    The first property follows from a simple Chernoff bound and a union bound.
    The second property for the unnormalized adjacency matrix is proven in
    \cite{NonNormalizedSpectra} and for the normalized adjacency matrix it follows from ~\cite[Theorem 3.6]{chung2004eigenvalues}.
    The third property is standard and follows from a simple union bound.
\end{proof}

\begin{theorem} \label{thm:typical-props}
    Let $n\geq 10$ be an integer and let $p \in \left (\tfrac{\log \left( n \right)}{\sqrt{n}} ,1 \right )$. If $G \in \Typ$ then it has the following
    properties.
    \begin{enumerate}
	\item \label{pr:evec_diff}
	    The top eigenvectors of $A$ and $\wh A$ are close to $\vec{1}$,
	    in the sense that
	    \[
		\Vert v_1(A) - \vec{1} \Vert_2 \le 2 \frac{\log \left( n \right)}{\sqrt{np}}
		\qquad \text{and}\qquad
		\Vert v_1(\wh A) - \vec{1} \Vert_2 \le \frac{2}{p} \frac{\log \left( n \right)}{\sqrt{n}}
	    \]
	    for all $n$ large enough.
	\item \label{pr:mat-proj}
	    If $S\subseteq [n]$, $P_S$ is the coordinate projection onto
	    $S$, and $Q_S$ is the projection onto the space orthogonal to
	    $\vec{1}_S$, then
	    \[
		\Vert Q_S P_S A \Vert_2 \le 2 \sqrt{\frac n p} \log \left( n \right)
		\qquad \text{and}\qquad
		\Vert Q_S P_S \wh A \Vert_2 \le \frac{2}{p} \frac{\log \left( n \right) }{\sqrt{n}}
	    \]
	    for all $n$ large enough.
	\item \label{pr:subtracting-mean}
	    The symmetric normalized adjacency matrix $\wh A$ is closely
	    approximated by $\tfrac{1}{np} A$ on vectors far from $\vec{1}$,
	    in the following specific sense.
	    For any subset $S \subseteq \left[ n \right]$ and any $x \in \R^n$ such that $\vert\langle x,\vec{1}\rangle\vert\le
	    \alpha$, we have
	    \[
	\Vert Q_S P_{S} \wh A x - \tfrac{1}{np} Q_S P_{S} A x\Vert_2
	\le p^{-5/2} \frac{\left( \log \left( n \right) \right)^2 + \alpha \sqrt{n} \log \left( n \right)}{n}.
	    \]
	\item \label{pr:ev2-bound}
	    The second eigenvalue of the symmetric normalized adjacency matrix $\wh A$ is not too small, specifically
	    \[
		\lambda_2(\wh A) \ge \left( 1 - o \left( 1 \right) \right) \frac{1-p}{16\sqrt{np}}.
	    \]
	\item \label{pr:large-entries}
	    Let $v$ be an eigenvector of $\wh A$
	    corresponding to the eigenvalue $\lambda$. Define the set
	    \[
		S := \left\{ i \in \left[ n \right] \, \middle| \, \left|v \left( i \right) \right| < \alpha \right\}.
	    \]
	    We then have
	    \[
		\Vert P_S v \Vert_2 \geq \frac{1}{3} \left( \frac{|\lambda|}{\lambda_2} - \frac{\log (n)}{\alpha^4
	    \lambda_2 np}\right) .
	\]
	In particular, if $\lambda = \Theta(\lambda_2)$, $\lambda = \Omega (n^{-1/2})$, and $\alpha \ge
	\frac{\log (n)}{(np)^{1/8}}$, then $\Vert P_S v \Vert_2 = \Omega(1)$.
    \end{enumerate}
\end{theorem}
The proof of \prettyref{thm:typical-props} is given by a series of lemmas in
\prettyref{app:typical}.

For a constant $C$, let $\cA_n^C$ denote the family of graphs $G$ on $n$ vertices having the following property:
All of the eigenvectors $v$ of the adjacency matrix $A_G$ satisfy
$$
\frac{ \Vert v \Vert_{\infty} }{ \Vert v \Vert_2} \leq  \left( \log \left( n \right) \right)^C / \sqrt{n}.
$$
The next result will be useful to us in proving a delocalization result for the eigenvectors of the unnormalized adjacency matrix in Section~\ref{sec:adj_union_bound}.

\begin{theorem}\label{thm:infty_bound}[\cite{erdHos2013spectral},~Theorem~2.16]
There exists a finite constant $C_{\infty}$ such that 
for any fixed $p \in (0,1)$ we have 
$\P\left( G\left(n,p\right) \in \cA_n^{C_{\infty}} \right) \to 1$ 
as $n \to \infty$. 
\end{theorem}
We note that~\cite[Theorem~2.16]{erdHos2013spectral} is a stronger and more general result; however, the theorem above is sufficient for our purposes. 
The results of ~\cite[Theorem~2.16]{erdHos2013spectral} show that $C_{\infty}$ can be taken to be any constant greater than $4$.

\section{From the spectral gap to delocalization of the second eigenvector}
In this section we show how the proof of Theorem~\ref{thm:main} can be reduced to a question about the entries of the second eigenvector of the normalized Laplacian. 
We prove the following proposition, which gives a sufficient condition for the addition of an edge to decrease the spectral gap.

\begin{proposition} \label{propsec3}
Let $G=(V,E) \in \Typ$. Let $f:V \to \R$ denote the eigenvector of $\nL_G$ corresponding to the eigenvalue $\lambda_2 \left( \nL_G
 \right)$, normalized such that $\left\| f \right\|_2 = 1$.  Let $u,v \in V$ be
 two vertices that are not connected by an edge, i.e., $\left\{u,v\right\}
 \notin E$.  Denote by $G_+ = \left( V, E_{+} \right)$ the graph obtained from
 $G$ by adding an edge between $u$ and $v$, i.e., $E_{+} := E \cup \bigl\{
     \{ u,v\} \bigr \}$. Then
$$
\frac{1}{\sqrt{np}}  \left ( f(u)^2 + f(v)^2 \right ) + (np)^{-2} <  c f(u)f(v) ~ \Rightarrow ~ \lambda_2 \left( \nL_{G} \right) > \lambda_2 \left( \nL_{G_+} \right),
$$
where $c>0$ is a universal constant. 

In particular, for every fixed $p \in (0,1)$ there exists a constant $C_p$ such that for all $n>C_p$ the following holds: 
if $\frac{1}{n^{0.51}} \leq |f(u)|, |f(v)| \leq \frac{1}{n^{0.49}}$ and $f(u)f(v) > 0$, 
then $\lambda_2 \left( \nL_{G} \right) > \lambda_2 \left( \nL_{G_+} \right)$. 
\end{proposition}

The main theorem of the paper follows easily from the above proposition together with the delocalization result described in Theorem~\ref{thm:deloc}. 

\medskip 

\begin{proof}\textbf{of Theorem~\ref{thm:main}.} 
Let $G \sim G(n,p)$. For a constant $c$ consider the event
$$
E_c := \left \{G \in \Typ \right \} \cap \left \{ \frac{1}{n} \# \left\{ i \in V : \left| f \left( i \right) \right| \in \left [ \frac{1}{n^{0.51}}, \frac{1}{n^{0.49}} \right ]  \right\} \geq c \right \},
$$
where $f : V \to \R$ is the eigenvector of $\nL_G$ corresponding to the eigenvalue $\lambda_2 \left( \nL_G \right)$, normalized such that $\left\| f \right\|_2 = 1$, just as in Proposition~\ref{propsec3}. 
According to Theorems~\ref{thm:deloc} and~\ref{thm:typical}, there exists $c>0$ depending only on $p$ such that $\P(E_c) \to 1$ as $n \to \infty$; 
in fact, by Theorem~\ref{thm:deloc} we can take $c = 1/2 - \eta$ for any constant $\eta > 0$. 
Define 
\begin{align*}
 J_+ &:= \left\{ i \in V : f \left( i \right) \in \left[ \frac{1}{n^{0.51}}, \frac{1}{n^{0.49}} \right] \right\} \\  
 \text{and} \qquad 
 J_- &:= \left\{ i \in V : - f \left( i \right) \in \left[ \frac{1}{n^{0.51}}, \frac{1}{n^{0.49}} \right] \right\}. 
\end{align*}
Whenever $E_c$ holds, we must have $\left| J_+ \right| + \left| J_- \right| \geq nc$. 
Let $F_+ := \binom{J_+}{2}$ be the set of possible edges between vertices in $J_+$, 
and define $F_- := \binom{J_-}{2}$ similarly. 
Since $G \in \Typ$, by Property~\ref{pr:discrepancy} in the definition of $\Typ$, we have 
\[
 \left| F_+ \setminus E \left( J_+ \right) \right| \geq \left( 1 - p \right) \binom{\left|J_+ \right|}{2} - n^{3/2}
\]
and
\[
 \left| F_- \setminus E \left( J_- \right) \right| \geq \left( 1 - p \right) \binom{\left|J_- \right|}{2} - n^{3/2}. 
\]
For the same reason, the number of edges in $G$ is at least $p \binom{n}{2} - n^{3/2}$, 
so the number of ``nonedges'' is at most $\left( 1 - p \right) \binom{n}{2} + n^{3/2}$. 
By Proposition~\ref{propsec3}, assuming that $n>C_p$, 
we have for every 
$(u,v) \in \left( F_+ \setminus E \left( J_+ \right) \right) \cup \left( F_- \setminus E \left( J_- \right) \right)$ 
that 
$\lambda_2 \left( \nL_{G} \right) > \lambda_2 \left( \nL_{G_+} \right)$, 
where $G_+$ is the graph obtained from $G$ by adding an edge between $u$ and $v$. 
Therefore 
\begin{align*}
 a_{-} \left( G \right) 
&\geq \frac{(1-p) \left[ \binom{\left| J_+ \right|}{2} + \binom{\left| J_- \right|}{2} \right] - 2 n^{3/2}}{\left( 1 - p \right) \binom{n}{2} + n^{3/2}} 
\geq \frac{(1-p)  \left[ \left( \frac{\left| J_+ \right| + \left| J_- \right|}{2}  \right)^2- \left( \left| J_+ \right| + \left| J_- \right| \right) \right] - 2 n^{3/2}}{\left( 1 - p \right) \binom{n}{2} + n^{3/2}} \\
&\geq \frac{(1-p)  \left[ \left( \frac{nc}{2}  \right)^2 - n \right] - 2 n^{3/2}}{\left( 1 - p \right) \binom{n}{2} + n^{3/2}} 
\geq \frac{c^2}{2} - o \left( 1 \right)
\end{align*}
as $n \to \infty$, which concludes the proof. 
\end{proof}

For the proof of Proposition~\ref{propsec3} we use the following lemma, which holds for any finite graph and follows from elementary computations. 

\begin{lemma}\label{lem:adding_edge}
 Let $G = \left(V,E\right)$ be a finite graph, and let $f$ denote the
 eigenvector of $\nL_G$ corresponding to the eigenvalue $\lambda_2 \left( \nL_G
 \right)$, normalized such that $\left\| f \right\|_2 = 1$.  Let $u,v \in V$ be
 two vertices that are not connected by an edge, i.e., $\left\{u,v\right\}
 \notin E$.  Denote by $G_+ = \left( V, E_{+} \right)$ the graph obtained from
 $G$ by adding an edge between $u$ and $v$, i.e., $E_{+} := E \cup \left\{
     \left\{ u,v\right\} \right\}$.  Define the quantity 
$p_f := \langle f, \evonep \rangle$, 
%      \begin{align*}
%     p_f &:= \langle f, \evonep \rangle,
%      \end{align*}
     the projection of $f$ onto the top eigenvector of $G_+$. If
\begin{equation}\label{eq:test}
\begin{split}
 p_f^2 \lambda_2 \left( \nL_G \right) + 2 \left( 1 - \lambda_2 \left( \nL_G \right) \right) \left\{ \frac{\sqrt{d_u + 1} - \sqrt{d_u}}{\sqrt{d_u + 1}} f\left( u \right)^2 + \frac{\sqrt{d_v + 1} - \sqrt{d_v}}{\sqrt{d_v + 1}} f\left( v \right)^2 \right\} \\
< \frac{2 f\left( u \right) f \left( v \right)}{\sqrt{d_u + 1}\sqrt{d_v + 1}},
\end{split}
\end{equation}
then
\[
\lambda_2 \left( \nL_{G} \right) > \lambda_2 \left( \nL_{G_+} \right),
\]
i.e., the spectral gap decreases by adding an edge between $u$ and $v$.
\end{lemma}

\begin{remark}
Note that this result holds even if $G$ is disconnected. In that case $\lambda_2 \left( \nL_G \right) = 0$, so the spectral gap clearly cannot decrease. This is not in contradiction with the lemma above; when $\lambda_2 \left( \nL_G \right) = 0$, the inequality~\eqref{eq:test} cannot hold.
\end{remark}

\medskip

\begin{proof}
 Let $D_+$ denote the diagonal matrix containing the degrees of the vertices in
 $G_+$ on the diagonal; the degrees are $d_i^+ = d_i$ for all $i \in V
 \setminus \left\{u,v\right\}$, and $d_u^+ = d_u + 1$ and $d_v^{+} = d_v + 1$.
 The first eigenvector of $\nL_G$ is $\evone := D^{1/2} \mathds{1}$,
 while the first eigenvector of $\nL_{G_+}$ is $\evonep := D_+^{1/2}
 \mathds{1}$.
By the variational characterization of eigenvalues, 
using also that 
the first eigenvalue of $\nL_{G_+}$, corresponding to the eigenvector $\evonep$, is
$0$, we have 
\begin{equation}\label{eq:lambda2_est}
    \lambda_2 \left( \nL_{G_+} \right) = \min_{\substack{x\\ \langle x,
	    \evonep\rangle =
    0}} \frac{x^T \nL_{G_+} x}{x^T x} \leq \frac{f_{\perp}^T \nL_{G_+} f_{\perp}}{f_\perp^T f_\perp} = \frac{f^T \nL_{G_+} f}{1 - p_f^2},
\end{equation}
where $f_\perp$ denotes the projection of $f$ onto the subspace orthogonal to
$\evonep$, and recall that $f$ is a unit vector.
A straightforward calculation---which requires a lot of bookkeeping; see~\prettyref{app:calc} for details---tells us that the expression for $f^T \nL_{G_+} f$ simplifies to the following:
\begin{equation}\label{eq:calc_final}
 \begin{split}
 f^T \nL_{G_+} f &= \lambda_2(\nL_G) + 2 \left( 1 - \lambda_2 \left( \nL_G \right) \right) \left\{ \frac{\sqrt{d_u + 1} - \sqrt{d_u}}{\sqrt{d_u + 1}} f\left( u \right)^2 + \frac{\sqrt{d_v + 1} - \sqrt{d_v}}{\sqrt{d_v + 1}} f\left( v \right)^2 \right\} \\
&\quad - \frac{2f(u)f(v)}{\sqrt{(d_u+1)(d_v+1)}}.
 \end{split}
\end{equation}
By~\eqref{eq:lambda2_est} we have that
\begin{equation}\label{eq:suff}
 \frac{f^T \nL_{G_+} f}{1 - p_f^2} < \lambda_2 \left( \nL_G \right)
\end{equation}
implies that $\lambda_2 \left( \nL_{G_+} \right) < \lambda_2 \left( \nL_G \right)$. By plugging in our expression for $f^T \nL_{G_+} f$ in~\eqref{eq:calc_final}, we get that~\eqref{eq:suff} is equivalent to~\eqref{eq:test}. 
\end{proof}

We are now ready to prove the main proposition of the section. 

\medskip 

\begin{proof}\textbf{of Proposition~\ref{propsec3}.} 
Using Lemma~\ref{lem:adding_edge}, we are interested in a sufficient condition for the inequality~\eqref{eq:test} to hold true. 
% \begin{equation} \label{eq:prevlem}
% \begin{split}
%  p_f^2 \lambda_2 \left( \nL_G \right) + 2 \left( 1 - \lambda_2 \left( \nL_G \right) \right) \left ( \frac{\sqrt{d_u + 1} - \sqrt{d_u}}{\sqrt{d_u + 1}} f\left( u \right)^2 + \frac{\sqrt{d_v + 1} - \sqrt{d_v}}{\sqrt{d_v + 1}} f\left( v \right)^2 \right ) \\
% < \frac{2 f\left( u \right) f \left( v \right)}{\sqrt{d_u + 1}\sqrt{d_v + 1}}.
% \end{split}
% \end{equation}
Since $G \in \Typ$, we have by definition that
$$
\forall v \in V, \ np - \log( n ) \cdot \sqrt{np} \le d_v \le np +
		\log( n ) \cdot \sqrt{np}. 
$$
Under the assumption that $n$ is large enough so that $\log( n ) \sqrt{np} < np/2$, we get that
\begin{equation} \label{eq:typcond1}
\forall v \in V, \ np/2 \le d_v \le 2np. 
\end{equation}
This gives us that
\begin{equation} \label{eq:degest1}
\frac{\sqrt{d_u + 1} - \sqrt{d_u}}{\sqrt{d_u + 1}} \leq \frac{1}{np}, ~ \frac{\sqrt{d_v + 1} - \sqrt{d_v}}{\sqrt{d_v + 1}} \leq \frac{1}{np} \mbox{ and } \frac{1}{\sqrt{d_u + 1}\sqrt{d_v + 1}} > \frac{1}{4np}.
\end{equation}
Since $f^T \evone = 0$, the length of the projection of $f$ onto $\evonep$ is
\begin{multline*}
   p_f := \left\langle f , \frac{1}{\sqrt{\sum_{i \in V}
	{d^+_i}}}D_+^{\tfrac{1}{2}}\mathds{1}\right\rangle \\
\begin{aligned}
&= \frac{1}{\sqrt{\sum_{i \in V} d^+_i}}
    \left(\left( \sum_{i \in V} f(i) \sqrt{d_i} \right)
    - f(u)\sqrt{d_u} - f(v) \sqrt{d_v} + f(u)\sqrt{d_u+1} +
f(v)\sqrt{d_v+1}\right)\\
&= \frac{1}{\sqrt{2+ \sum_{i \in V} d_i}}
\left\{ f(u)(\sqrt{d_u + 1} - \sqrt{d_u}) + f(v)(\sqrt{d_v + 1}
- \sqrt{d_v}) \right\}, 
\end{aligned}
\end{multline*}
which, together with equation~\eqref{eq:typcond1} gives that 
$$
|p_f| \leq (np)^{-3/2} \left ( \left| f(u) \right| + \left| f(v) \right| \right ) \leq 2 (np)^{-3/2}.
$$
The above equation, together with~\eqref{eq:degest1}, plugged into~\eqref{eq:test} gives that a sufficient condition for~\eqref{eq:test} to hold is that 
\begin{equation}\label{eq:a_suff_cond}
8 \left( np \right)^{-2} + 4 \left( 1 - \lambda_2 \left( \nL_G \right) \right) \left ( f(u)^2 + f(v)^2 \right ) < f(u)f(v).
\end{equation}
Since $G \in \Typ$, by definition we also have that
$$
\left | 1 - \lambda_2 \left( \nL_G \right) \right |  = \left| \lambda_2(\wh A) \right| \leq \frac{8}{\sqrt{p n}}, 
$$
and so~\eqref{eq:a_suff_cond} is implied by the inequality
\[
8 \left( np \right)^{-2} + 32 \left( np \right)^{-1/2} \left ( f(u)^2 + f(v)^2 \right ) < f(u)f(v), 
\]
which concludes the proof. 
\end{proof}

\section{Delocalization of the second eigenvector}

In this section we prove our delocalization result stated in
Theorem~\ref{thm:deloc}. As a warm-up, we first prove an analogous result for
the adjacency matrix $A_G$ of $G=G(n,p)$ which contains most of the main
ideas of the proof for the normalized case, but is somewhat simpler. We then
present the proof for the normalized Laplacian, which carries with it some
extra difficulties.

Before we move on to these proofs, we need to collect a few auxiliary results related to small ball concentration bounds for sums of independent random variables. 
We present these in the next subsection.

\subsection{Small ball concentration estimates}\label{sec:useful}
Consider a vector whose entries are independent sums of independent scaled Bernoulli random variables. 
Our proof hinges on showing an upper bound for the probability that such a vector has small norm. 
To do this, we rely on a previous Littlewood-Offord-type result and also on a theorem of Rudelson and Vershynin. 

The following definition is natural for our purposes. 
\begin{definition}
For a real random vector $Z \in \R^n$ and $t \ge 0$, define the concentration function
\[
	\conc(Z,t) := \max_{q \in \R^n} \P \left[ \Vert Z - q \Vert_2 \le t \right].
\]
\end{definition}
This function measures the largest probability that a random vector lands in a ball of fixed radius. 

For a single entry of a vector, we use the following lemma to bound the 
concentration function.
\begin{lemma}\label{lem:interval_prob}
	Let $X = \sum_{i \in [n]} a_i \beta_i$, where the $\beta_i \sim \Ber(p)$ are
	independent Bernoulli random variables with expectation $p$.
	There exists an absolute constant $C < \infty$ such that if $|a_i| \ge 1$ for at least
	$m$ indices $i \in [n]$, then for all $r \ge 1$,
	\[
	    \conc(X,r) \le \frac{Cr}{\sqrt{m p (1-p)}}.
	\]
\end{lemma}
This is a simple generalization of Erd\H{o}s's strengthening of the Littlewood-Offord theorem~\cite{littlewood1943number,erdos1945lo}. We provide the proof (based on an idea of Hal\'asz~\cite{halasz1977estimates}; see also~\cite{NguyenVu}) for completeness. 

\medskip 

\begin{proof} 
It suffices to prove the statement for $r = 1$, as the dependence on $r$ follows from a union bound. 
By a standard computation (see e.g. \cite[Lemma 6.2]{NguyenVu}) we have that
$$
\conc(X,1) \leq C \int_{-1}^{1} |\E [\exp(it X)]| dt
$$
for a universal constant $C>0$. Write
$$
J := \{j : \left| a_j \right| \ge 1 \}.
$$
By the independence of the $\beta_j$'s and using H\"{o}lder's inequality we have
\begin{align*}
\int_{-1}^{1} |\E [\exp(it X)]| dt ~& = \int_{-1}^{1} \left |\prod_{j \in [n]} \E [\exp(it a_j \beta_j)] \right | dt \\
& \leq \int_{-1}^{1} \left |\prod_{j \in J} \E [\exp(it a_j \beta_j)] \right | dt \\
& \leq   \prod_{j \in J} \left ( \int_{-1}^{1} \left | \E [\exp(it a_j \beta_j)] \right |^m dt \right )^{1/m},
\end{align*}
where $m$ is the cardinality of $J$. The proof would therefore be concluded by proving that for all $j \in J$, one has
\begin{equation} \label{ntsLO}
\int_{-1}^{1} \left | \E [\exp(it a_j \beta_j)] \right |^m dt \leq \frac{C'}{\sqrt{m p(1-p)}}
\end{equation}
for a universal constant $C'$. 
We have
$$
\E [\exp(it a_j \beta_j)] = \left( 1 - p \right) + p \exp(i t a_j ),
$$
so
$$
\left | \E [\exp(it a_j \beta_j)] \right |^2 = 1 - 2 p(1-p) (1-\cos(a_j t)),
$$
and thus (substituting $w=a_j t$) we have 
$$
\int_{-1}^{1} \left | \E [\exp(it a_j \beta_i)] \right |^m dt = \frac{1}{\left| a_j \right|} \int_{-\left| a_j \right|}^{\left| a_j \right|}  \bigl (1 - 2 p(1-p) (1-\cos(w)) \bigr )^{m/2} dw.
$$
Using the periodicity of $\cos(x)$, its monotonicity in the interval $[0,\pi]$, and also using the fact that $\left| a_j \right| \geq 1$, 
we have
$$
\frac{1}{\left| a_j \right|} \int_{-\left| a_j \right|}^{\left| a_j \right|}  \bigl (1 - 2 p(1-p) (1-\cos(w)) \bigr )^{m/2} dw \leq 4 \int_{-\pi/2}^{\pi/2} \bigl (1 - 2 p(1-p) (1-\cos(w)) \bigr )^{m/2} dw.
$$
Next, using the fact that $1-\cos(x) \geq x^2/8$ for $x \in [-\pi/2, \pi/2]$, we have
$$
\int_{-\pi/2}^{\pi/2} \bigl (1 - 2 p(1-p) (1-\cos(w)) \bigr )^{m/2} dw \leq \int_{-\pi/2}^{\pi/2} \bigl (1 - \tfrac 1 4 p(1-p) w^2 \bigr )^{m/2} dw \leq \frac{C''}{\sqrt{m p(1-p)}}
$$
for some constant $C''>0$. Putting the last displays together gives~\eqref{ntsLO}. %, which is what we needed to show. 
\end{proof}

We also use the following result, which roughly states that 
if $X = \left( X_1, \ldots, X_n \right)$ is a random vector with independent coordinates 
and the distributions of the $X_i$ are well spread on the line, 
then the distribution of a linear image of $X$ by a certain linear transformation is also well-spread. 
This result will be used in conjunction with the Littlewood-Offord-type lemma above. 
It is a simple analog (but not a special case) of \cite[Corollary 1.5]{rudelson2014smallball}.

\begin{lemma} \label{lem:RV}
Let $1 \leq d < n$ be integers. Suppose that $X = \left(X_1,\ldots, X_d \right)$ is a random vector where the $X_i$ are independent real-valued random variables, and that $t,q\ge 0$ are such that for all $i \in [d]$,
\[
	\conc(X_i, t) \le q.
\]
Suppose also that $|X_i| \leq K$ almost surely for all $i \in [d]$ and for some $K>0$.
Let $T$ be a linear isometric embedding of $\R^d$ in $\R^n$ and let $H \subset \R^n$ be an $(n-1)$-dimensional subspace. Let $P_H$ denote the orthogonal projection from $\R^n$ onto $H$. Then there exists an absolute constant $C < \infty$ such that
\[
	\conc(P_H T X, t\sqrt{d}) \le (Cq)^d (K/t+1) \sqrt{d}.
\]
\end{lemma}

\begin{proof}
By rescaling, we may clearly assume that $t=1$ and replace $K$ by $K'=K/t$. Let $Y_1,...,Y_d$ be independent random variables uniformly distributed on $[-1,1]$ and let $Y=(Y_1,...,Y_d)$. Define $Z = X+Y$, denote by $f_i$ the density of $Z_i$ and by $f$ the density of $Z$. Note that we have
$$
f_i(x) = \tfrac 1 2 P( |X_i-x| \leq 1 ) \leq \conc(X_i, 1) \le q, ~~ \forall x \in \R
$$
and therefore
\begin{equation} \label{eq:boundf}
f(x) \leq q^d, ~~ \forall x \in \R^d.
\end{equation}
Denote by $V$ the image of the operator $T$, and by $P_V$ the orthogonal projection onto $V$. Suppose for now that $V \nsubseteq H$ (the other case is in fact simpler). Define $\tilde H = P_H V$ and $W = V \cap H$. By dimension considerations, there exists a unit vector $v \perp W$ such that $V = \mathrm{sp} \left (W \cup \{v\}\right )$.

Fix a point $x \in \R^n$. By the triangle inequality and since almost surely, $\Vert Y \Vert \leq \sqrt{d}$, we have
\begin{equation} \label{eq:PHconc}
\P \left (\Vert P_H T X - x \Vert < \sqrt d \right ) \leq \P \left ( \Vert P_H T Z - x \Vert \leq 2 \sqrt{d} \right  ).
\end{equation}
Now, since $P_W$ is a contraction and $|\det T| = 1$, we have 
\begin{align*}
\P \left ( \Vert P_H T Z - x \Vert \leq 2 \sqrt{d} \right  ) ~& \leq \P \left ( \Vert P_W T Z - P_W x \Vert \leq 2 \sqrt{d} \right  ) \\
& = \int_{ \left \{ z \in V: ~ \Vert P_W z- P_W x \Vert \leq 2 \sqrt d \right \} } f \left (T^{-1} z \right ) dz \\
& = \int_{ \left \{ y \in W: ~ \Vert y- P_W x \Vert \leq 2 \sqrt d \right \} } \int_{\R} f \left (T^{-1} (y + sv) \right ) ds dy.
\end{align*}
Next, since we have by assumption $|X_i| \leq K'$ almost surely and since $|Y_i| \leq 1$, we have $f \left ( T^{-1} (w + sv) \right ) = 0$ for all $w \in W$ whenever $|s| > (K'+1) \sqrt{d}$. Plugging this fact, together with equation \eqref{eq:boundf}, into the last inequality yields
\begin{align*}
\P \left ( \Vert P_H T Z - x \Vert \leq 2 \sqrt{d} \right  )~& \leq \mathrm{Vol}_W \left (\left \{ y \in W; ~ \Vert y- P_W x \Vert \leq 2 \sqrt d \right \} \right ) q^d (K'+1) \sqrt{d} \\
& \leq (4q)^d (K'+1) \sqrt{d}
\end{align*}
where $\mathrm{Vol}_W$ denotes the $(d-1)$-dimensional Lebesgue measure in $W$, and in the second inequality we have used a standard estimate related to the volume of the $(d-1)$-dimensional unit ball. Together with equation \eqref{eq:PHconc} we conclude that
$$
\conc(P_H T X, t\sqrt{d}) \le (4q)^d (K'+1) \sqrt{d}
$$
which finishes the proof for that case that $V \nsubseteq H$. For the (simpler)
case that $V \subseteq H$ we just plug in equation \eqref{eq:PHconc} with \eqref{eq:boundf} and with the same estimate for the volume of the Euclidean ball that we have used above.
\end{proof}

% Let $\cA_n$ denote the family of graphs $G$ having the following property: All of the eigenvectors $v$ of the adjacency matrix $A_G$ satisfy
% $$
% \frac{ \Vert v \Vert_{\infty} }{ \Vert v \Vert_2} \leq  \left( \log \left( n \right) \right)^2 / \sqrt{n}.
% $$
% The next result will be useful to us in proving a delocalization result for the eigenvectors of the non-normalized adjacency matrix.
% 
% \begin{theorem}\label{thm:infty_bound}[\cite{tran2010eigenvectors}, Theorem 1.16]
% For any fixed $p \in (0,1)$ we have $\P(G(n,p) \in \cA_n) \to 1$ as $n \to \infty$.
% \end{theorem}

\subsection{Delocalization}

\subsubsection{Delocalization of eigenvectors of the adjacency matrix}\label{sec:adj_union_bound}

\begin{theorem}\label{thm:adj-main}
    Fix $p \in \left( 0, 1 \right)$,
    and let $G$ be an instance of $G(n,p)$.
    For any constant $\eta > 0$
    there exists a finite positive constant $C = C \left( \eta, p \right)$ such that
    asymptotically almost surely all eigenvectors of $A_G$ (normalized to have unit $\ell_2$-norm)
    have at least $\left( 1/2 - \eta \right) n$ entries
    of magnitude at least $\frac{1}{\sqrt{n} \left( \log \left( n \right) \right)^C}$.
\end{theorem}

For the first eigenvector $v_1$ a stronger statement is known; see~\cite{mitra2009entrywise}. Consequently, we focus our attention on the eigenvectors $v_2, \dots, v_n$, which are orthogonal to $v_1$.

The following lemma is the main step towards proving the theorem above.
Recall that $\PT$ denotes the distribution of an instance of $G = G(n,p)$ conditioned on $G \in \Typ$  (see Definition~\ref{def:typical}).
Recall that $\cA_n^C$ denotes the family of graphs $G$ on $n$ vertices such that all of the eigenvectors of $A_G$, normalized to have unit $\ell_2$-norm, have infinity-norm bounded by
$\left( \log \left( n \right) \right)^C / \sqrt{n}$. 
Recall also from Theorem~\ref{thm:infty_bound} that there exists a finite constant $C_{\infty}$ such that 
$\{G(n,p) \in \cA_n^{C_{\infty}} \}$ occurs with probability tending to $1$ as $n \to \infty$ for any fixed $p \in \left( 0, 1 \right)$. 
In what follows, $C_{\infty}$ always denotes this constant. 
\begin{lemma} \label{lem:adj-fixed-W}
    Fix $p \in \left( 0, 1 \right)$ and $\eps \in \left( 1/4, 1/2 \right)$.
    Let $W \subseteq \left[ n \right]$ be of size $\eps n$,
    and let $W^C := \left[n \right] \setminus W$.
    Let $\delta$ be such that $n^{-1/2 + 1/10} < \delta < 1/10$.
    Fix $j \in \left\{2, 3, \dots, n \right\}$.
    Recall that $v_j$ denotes the $j^{\text{th}}$ eigenvector of $A_G$.
    Then there exists a finite constant $C_p$, depending only on $p$, such that
    \begin{equation}\label{eq:bd_adj_fixed_W}
      \PT \left( \left| v_j \left( i \right) \right| \leq \frac{\delta}{\sqrt{n}} \ \text{ for all } \  i \in W^C \, \middle| \, G \in \cA_n^{C_{\infty}} \right)
      \leq
      \left( C_p \log \left( n \right) \right)^{\left( C_{\infty} + 1 \right) n}  \times \delta^{\left( 1- 2 \eps \right) n}.
    \end{equation}
\end{lemma}

\begin{proof}
    Our proof proceeds by a union bound over candidate eigenvectors.
    Let $\Omega_W \subset \R^n$ be the set of all vectors obeying the appropriate constraints,
    that is,
    \[
	    \Omega_W := \left\{
		    v \in \R^n \ \, \middle| \,
		    \ \Vert v \Vert = 1,
		    \ \Vert v \Vert_{\infty} \le \frac{\left( \log \left( n \right) \right)^{C_{\infty}}}{\sqrt n},
		    \ |v(i)| \le \frac{\delta }{\sqrt n} \ \ \ \forall i \in W^C
	    \right\}.
    \]
    Given $G \in \cA_n^{C_{\infty}}$, if $\left|v_j \left( i \right) \right| \leq \delta / \sqrt{n}$ for all $i \in W^C$, then $v_j \in \Omega_W$.
    We define a net $\Lambda_W$ over $\Omega_W$ with resolution $R := \delta / \sqrt{n}$ in the following way:
    \begin{multline*}
	\Lambda_W
	:= \left\{ x \in \R^n \, \middle| \,
	    x = R \cdot k,
	    \ k \in \Z^{n},
	    \ k_i = 0\ \ \ \forall 	i \in W^C,
	    \ |k_i| \le \frac{\left( \log \left( n \right) \right)^{C_{\infty}}}{R\sqrt{n}}\ \ \ \forall i \in W, \right. \\
	    \ \Vert x \Vert \in \left[ 1 - \delta, 1 + \delta \right]
		\Bigg\}.
    \end{multline*}
    The discretization $\Lambda_W$ has the
    property that for any $v \in \Omega_W$,
    there exists $x \in \Lambda_W$ such that
    $u := v - x \in \left[-\tfrac{\delta}{\sqrt n},
	\tfrac{\delta}{\sqrt n}\right]^n$.
    The cardinality of the net $\Lambda_W$ can be bounded from above by noting that
    for any $x \in \Lambda_W$,
    the coordinates of $x$ in $W^C$ are fixed, while the coordinates in $W$ can take on
    at most $2 \left( \log \left( n \right) \right)^{C_{\infty}} / \left( R \sqrt{n} \right) + 1$ values,
    and so
\begin{equation} \label{eq:netsize1}
|\Lambda_W|
	\le \left( \frac{2\left( \log \left( n \right) \right)^{C_{\infty}}}{R \sqrt n} + 1 \right)^{|W|}
	\le \left( \frac{3\left( \log \left( n \right) \right)^{C_{\infty}}}{\delta} \right)^{\epsilon n}.
\end{equation}

    For a vector $v \in \R^n$, define the event
    \[
	F_{W,v,A} := \left\{
	    Q_{W^C} P_{W^C} A v \in B \left( 0,\delta n^{1/2} \log \left( n \right) \right)
	\right\},
    \]
    where $P_{W^C}$ is the projection onto the coordinates of $W^C$,
    and $Q_{W^C}$ is the orthogonal projection onto the space orthogonal to $\vec{1}_{W^C}$.
    We claim that
    \begin{equation} \label{eq:FW1}
    G \in \Typ \mbox{ and } v_j \in \Omega_W \Rightarrow F_{W,v_j,A}.
    \end{equation}

    Indeed, by definition,
    $\left( A v_j \right) \left( i \right) = \lambda_j v_j \left( i \right)$
    for all coordinates $i \in \left[ n \right]$.
    Since $G \in \Typ$ and $j \geq 2$, we have $\left| \lambda_j \right| \leq 3 \sqrt{n}$.
    Since $v_j \in \Omega_W$, we have
    $\left| v_j \left( i \right) \right| \leq \delta / \sqrt{n}$ for all $i \in W^C$,
    and so $\left| \left( P_{W^C} A v_j \right) \left( i \right) \right| \leq 3 \delta$ for all $i \in W^C$.
    Note also that $\left( P_{W^C} A v_j \right) \left( i \right) = 0$ for all $i \in W$.
    Therefore $P_{W^C} A v_j \in B \left( 0, 3 \delta \sqrt{n} \right) \subseteq B \left( 0,\delta n^{1/2} \log \left( n \right) \right)$.
    Since $\Vert Q_{W^C} \Vert = 1$, it follows that the event $F_{W, v_j, A}$ holds, which establishes the implication in~\eqref{eq:FW1}.

    Next, for $v \in \Omega_W$, let $x \in \Lambda_W$ be the closest point in $\Lambda_W$
    such that $u := v - x \in \left[ - \delta / \sqrt{n}, \delta / \sqrt{n} \right]^n$ 
    (such an $x \in \Lambda_W$ exists; in case it is not unique, take one of the closest points arbitrarily).
    Then
    \begin{equation*}
	\Vert Q_{W^C} P_{W^C}  A x \Vert
	\le \Vert Q_{W^C} P_{W^C} A v\Vert + \Vert Q_{W^C} P_{W^C} A u\Vert
	\le \Vert Q_{W^C} P_{W^C} A v\Vert + 2 \delta \sqrt{n/p} \log n,
    \end{equation*}
    where the first inequality follows from the triangle inequality,
    and the second inequality follows from the Cauchy-Schwarz inequality,
    the fact that $\Vert u \Vert \le \delta$, and \prettyref{thm:typical-props}, part \ref{pr:mat-proj}.
    Therefore, if $G \in \Typ$ and $v_j \in \Omega_W$, then we have
    \begin{equation} \label{eq:HW1}
    F_{W, v_j, A} \mbox { holds } \Rightarrow \exists x \in \Lambda_W \mbox{ such that } H_{W,x,A} \mbox{ holds},
    \end{equation}
    where
    \[
	H_{W,x,A} := \left\{
	    Q_{W^C} P_{W^C} A x \in B\left(0, 3 \delta \sqrt{ \tfrac n p} \log \left( n \right) \right)
	\right\}.
    \]

    We now fix $x \in \Lambda_W$, and bound $\P \left[ H_{W,x,A} \right]$.
    Note that $x_i = 0$ for all $i \in W^C$, so we can write $x = P_W x$,
    where $P_W$ is the coordinate projection onto $W$.
    Define $Y := P_{W^C} A x$. Thus $Y_i = 0$ for $i \in W$,
    while for $i \in W^C$ we have
    $Y_i = \sum_{j \in W} A_{ij} x_j$,
    that is, $Y_i$ is a sum of scaled independent Bernoulli random variables.
    By design, for any $x \in \Lambda_W$ we have
    $\Vert x \Vert_{\infty} \le \frac{\left( \log \left( n \right) \right)^{C_{\infty}} }{\sqrt n}$
    and $\Vert x \Vert_2^2 \ge \left( 1 - \delta \right)^2 > \tfrac{3}{4}$,
    and so there are at least $\frac{n}{2 \left( \log \left( n \right) \right)^{2C_{\infty}}}$ entries of $x$ with magnitude at least $\frac{1}{2 \sqrt n}$.
    We can now apply \prettyref{lem:interval_prob} to $2 \sqrt{n} Y_i$
    with $m = \frac{n}{2 \left( \log \left( n \right) \right)^{2C_{\infty}}}$
    and $r = 12 \delta \sqrt{ \tfrac n p} \log \left( n \right)$ to get that
%     \begin{align*}
%       \conc \left( Y_i,\frac{6 \delta}{\sqrt{p}} \log \left( n \right) \right)
%       &= \conc \left( 2 \sqrt{n} Y_i, 12 \delta \sqrt{\tfrac n p } \log \left( n \right) \right) \\
%       &\leq \frac{C}{p \sqrt{1-p}} \delta \left( \log \left( n \right) \right)^3
%     \end{align*}
    \begin{equation*}
      \conc \left( Y_i,\frac{6 \delta}{\sqrt{p}} \log \left( n \right) \right)
      = \conc \left( 2 \sqrt{n} Y_i, 12 \delta \sqrt{\tfrac n p } \log \left( n \right) \right) 
      \leq \frac{C}{p \sqrt{1-p}} \delta \left( \log \left( n \right) \right)^{C_{\infty}+1}
    \end{equation*}
    for some finite universal constant $C > 0$.

    Furthermore, the random variables $\left\{ Y_i \right\}_{i \in W^C}$ are independent.
    This is because these random variables are functions of disjoint subsets of the random variables $\left\{ A_{ij} \right\}_{i,j \in \left[ n \right], i < j}$, 
    because we can write $Y = P_{W^C} A x = P_{W^C} A P_W x$ and since $x_i = 0$ for all $i \in W^C$, and $W$ and $W^C$ are disjoint.
    Thus, we can apply \prettyref{lem:RV} to $Y$, with
    $t = \frac{6 \delta}{\sqrt{p}} \log \left( n \right)$,
    $d = (1-\epsilon)n$,
    $K = n^3$,
    $q = \frac{C}{p \sqrt{1-p}} \delta \left( \log \left( n \right) \right)^{C_{\infty}+1}$, and
    $H = \vec{1}_{W^C}^\perp $ to get
	\begin{equation} \label{eq:HWbound1}
	\begin{split}
      \P \left[ H_{W,x,A} \right]
      \le & ~ \conc \left( Q_{W^C} Y, 3 \delta \sqrt{ \tfrac n p} \log \left( n \right) \right) \\
      \le & ~ \conc \left( Q_{W^C} Y, t \sqrt{d} \right)
      \le \left( \frac{C'}{p \sqrt{1-p}} \delta \left( \log \left( n \right) \right)^{C_{\infty}+1} \right)^{\left( 1 - \eps \right) n} \times n^6 %\nonumber
	\end{split}
	\end{equation}
    for a universal constant $C'>0$, where in the second inequality we used the fact that $\eps \in \left( 1/4, 1/2 \right)$.

    Finally, we take a union bound to arrive at our result:
    \begin{multline*}
      \PT \left( \left| v_j \left( i \right) \right| \leq \frac{\delta}{\sqrt{n}} \ \text{ for all } \  i \in W^C \, \middle| \, G \in \cA_n^{C_{\infty}} \right) \\
      \begin{aligned}
      & \stackrel{\eqref{eq:FW1}}{\le} \PT \left( F_{W, v_j, A} \, \middle| \, G \in \cA_n^{C_{\infty}} \right)
      \stackrel{\eqref{eq:HW1}}{\le} \PT \left( \cup_{x \in \Lambda_W} H_{W,x,A} \, \middle| \, G \in \cA_n^{C_{\infty}} \right) \\
      &\le 2 \P \left( \cup_{x \in \Lambda_W} H_{W,x,A} \right)
      \le 2 \left| \Lambda_W \right| \max_{x \in \Lambda_W} \P \left( H_{W,x,A} \right) \\
      & \stackrel{ \eqref{eq:netsize1} \wedge \eqref{eq:HWbound1} }{\le} 2 n^6 \left( \frac{3 \left( \log \left( n \right) \right)^{C_{\infty}}}{\delta} \right)^{\eps n}
   	\times \left( \frac{C'}{p \sqrt{1-p}}  \delta \left( \log \left( n \right) \right)^{C_{\infty}+1} \right)^{\left( 1 - \eps \right) n}  \\
      &= 2 n^6 \times 3^{\eps n} \times \left ( \frac{C'}{p \sqrt{1-p}} \right )^{\left( 1- \eps \right) n} \times \delta^{\left( 1 - 2 \eps \right) n} \times \left( \log \left( n \right) \right)^{\left( C_{\infty} + 1 - \eps \right) n} \\
      &\le \left( C_p \log \left( n \right) \right)^{\left( C_{\infty} + 1 \right) n}  \times \delta^{\left( 1- 2 \eps \right) n},
      \end{aligned}
    \end{multline*}
    where in the third inequality, which holds for $n$ large enough, we used Theorem \ref{thm:typical} and Theorem \ref{thm:infty_bound}.
    \end{proof}

Using this lemma we now prove \prettyref{thm:adj-main}.

\medskip

\begin{proof}
    Fix a constant $\eta > 0$, and let $\eps = 1/2 - \eta$.
    We apply \prettyref{lem:adj-fixed-W},
    and take a union bound over the possible subsets $W\subseteq\left[n\right]$
    (of which there are at most $2^n$)
    and over the possible eigenvectors.
    The lemma thus tells us that there exists a constant $C$
    such that,
    conditioned on $G \in \Typ$ and $G \in \cA_n^{C_{\infty}}$, the probability that there exists a subset $W \subseteq \left[ n \right]$ of size $\eps n$
    and an eigenvector $v_j$, with $j \geq 2$,
    such that $\left| v_j \left( i \right) \right| \leq \delta / \sqrt{n}$ for all $i \in W^C$
    is at most
    \[
      \left( C \log \left( n \right) \right)^{\left( C_{\infty} + 1 \right) n} \times \delta^{2\eta n}.
    \]
    Now choosing $\delta := \left( \log \left( n \right) \right)^{-\left( C_{\infty} + 1 \right)/\eta}$, we get that
    conditioned on $G \in \Typ$ and $G \in \cA_n^{C_{\infty}}$,
    the probability that there are not at least $\left( 1/2 - \eta \right) n$ entries of each eigenvector of $A$ of magnitude at least
    \[
      \frac{1}{\sqrt{n} \left( \log \left( n \right) \right)^{\left( C_{\infty} + 1 \right)/\eta}}
    \]
    is at most $\left( C / \log \left( n \right) \right)^{\left( C_{\infty} + 1 \right)n}$.
    Since $G \in \Typ \cap \cA_n^{C_{\infty}}$ asymptotically almost surely (by Theorems \ref{thm:typical} and \ref{thm:infty_bound}), we are done.
\end{proof}

\subsubsection{Delocalization of eigenvectors of the normalized adjacency matrix}

In this section we prove our main delocalization result, Theorem \ref{thm:deloc}.
Our proof for the normalized adjacency matrix also proceeds by a union bound over candidate eigenvectors.
However, it is slightly more involved than for the (unnormalized) adjacency matrix.
For one, the degree normalizations $D^{\smlhf}AD^{\smlhf}$ introduce correlations between the rows of the matrix.
Another major issue is the $\ell_\infty$ bound, Theorem~\ref{thm:infty_bound},
which is known to hold for the adjacency matrix,
but it is not known to hold for the normalized case.
Still, with the help of some additional technical lemmas and with a more careful choice of a net for candidate eigenvectors,
the proof proceeds more or less along the same lines.

The central lemma in the proof is the following.
\begin{lemma} \label{lem:fixed-W}
    For every $p \in \left( 0, 1 \right)$, there exists a constant $C_p > 0$ such that the following holds.
    Let $\eps \in \left( 1/4, 1/2 \right)$, let $W \subseteq \left[ n \right]$ be a subset of size $\eps n$, and let $W^C := \left[n \right] \setminus W$.
    Let $\delta$ be such that $n^{-1/10} < \delta < 1/100$ and let $n \geq C_p$ be an integer.
    Let $G$ be an instance of $G(n,p)$, and let $v_2$ denote the second eigenvector of $\wh{A}$, normalized such that $\left\| v_2 \right\| = 1$.
    Then we have
    \begin{equation}\label{eq:bd_fixed_W}
      \PT \left( \left| v_2 \left( i \right) \right| \leq \frac{\delta}{\sqrt{n}} \ \text{ for all } \  i \in W^C \right)
      \leq
      \left( C_p \log \left( n \right) \right)^{2n}  \times \delta^{\left( 1- 2 \eps \right) n}.
    \end{equation}
\end{lemma}

\begin{proof} \\
\textbf{Step 1:} For a vector $v \in \R^n$, define the set of indices
    \[
     S \left( v \right) := \left\{ i \in \left[ n \right] \, : \, \left| v \left( i \right) \right| \leq \frac{\log \left( n \right)}{n^{1/8}} \right\},
    \]
    and also
    \[
     S' \left( v \right) := \left\{ i \in \left[ n \right] \, : \, \left| v \left( i \right) \right| \leq \frac{2 \log \left( n \right)}{n^{1/8}} \right\}.
    \]
    Define the subset $\Omega_W \subset \R^n$ as follows:
    \[
	    \Omega_W := \left\{
		    v \in \R^n \, \middle| \,
		    \Vert v \Vert = 1,
		    \ \left| \left\langle v, \vec{1} \right\rangle \right| \leq \frac{2}{p} \frac{\log \left( n \right)}{\sqrt{n}},
		    \ \left| v \left( i \right) \right| \leq \frac{\delta}{\sqrt{n}} \ \ \ \forall i \in W^C,
		    \ \Vert v|_{S\left( v \right)} \Vert_2 \ge \frac{1}{10}
 	    \right\}.
    \]
    Let $\vec{1}_D := \mathrm{diag} ( \sqrt{d_1}, \dots, \sqrt{d_n} ) / \sqrt{\sum_{j} d_j}$ and recall that $v_1 ( \wh{A} ) = \vec{1}_D$.
    Note that if $G \in \Typ$ then by \prettyref{thm:typical-props}, part \ref{pr:evec_diff}, we have
    \[
      \left| \left \langle v_2, \vec{1} \right \rangle \right|
      \le \left| \left \langle v_2, \vec{1} - \vec{1}_D \right \rangle \right|
      + \left| \left \langle v_2, \vec{1}_D \right \rangle \right|
      \le \Vert v_2 \Vert \Vert \vec{1} - \vec{1}_D \Vert + 0
      \le \frac{2}{p} \frac{\log \left( n \right)}{\sqrt{n}}.
    \]
    Furthermore, by Properties~\ref{pr:ev2-bound}~and~\ref{pr:large-entries} of Theorem~\ref{thm:typical-props}, we have that if $G \in \Typ$ then $\Vert v_2|_{S\left( v_2 \right)} \Vert_2 \ge \frac{1}{10}$.
    Thus if $G \in \Typ$ then
    \begin{equation} \label{eq:omega2}
    \left|v_2 \left( i \right) \right| \leq \delta / \sqrt{n}, ~ \forall i \in W^C \Rightarrow v_2 \in \Omega_W.
    \end{equation}

    \medskip

    \textbf{Step 2:} In this step we construct a net $\Lambda_W$ over
    $\Omega_W$ for candidate eigenvectors. Again, this is a net with
    resolution $R = \delta / \sqrt{n}$. However, the construction is a bit more
    involved than in the unnormalized case, because we have to overcome the lack of
    an analogue of Theorem~\ref{thm:infty_bound} (the bound on
    $\|v\|_{\infty}$).
    We define
    \begin{multline*}
	\Lambda_W
	:= \left\{ x \in \R^n \, \middle| \,
	    x = R \cdot k,
	    \ k \in \Z^{n},
	    \ k_i = 0\ \ \ \forall 	i \in W^C,
	    \ \left| \left \langle x, \vec{1} \right \rangle \right| \leq \frac{3}{p} \frac{\log \left( n \right)}{\sqrt{n}},
	    \right. \\
	    \ \Vert x \Vert \in \left[ 1 - 2 \delta, 1 + 2 \delta \right],
	    \ \Vert x|_{S' \left( x \right)} \Vert_2 \geq \frac{1}{20}
		\Bigg\}.
    \end{multline*}
    We claim that this net has the following property:
    \begin{equation} \label{eq:netprop2}
	\forall v \in \Omega_W, ~ \exists x \in \Lambda_W \mbox{ such that } u := v - x \in \left[ - \tfrac{4 \delta}{\sqrt{n}}, \tfrac{4 \delta}{\sqrt{n}} \right]^n.
    \end{equation}

    To see this, given $v \in \Omega_W$, first define $x' \in \R^n$ by setting
    $x' \left( i \right) = 0$ for $i \in W^C$, and $x' \left( i \right) = v \left( i \right)$ for $i \in W$.
    Since $\left| v \left( i \right) \right| \le \delta / \sqrt{n}$ for all $i \in W^C$, we have
    \[
      \left| \left\langle x', \vec{1} \right \rangle - \left\langle v, \vec{1} \right \rangle \right| \leq  (1 - \eps) \delta.
    \]
    So the inner product $\left\langle x', \vec{1} \right \rangle$ might have large magnitude,
    but this can be ``corrected for'' and made close to zero by changing the coordinates of $x'$ in $W$ by at most $\tfrac{ (1-\eps) \delta}{\eps \sqrt{n}}$ each.
    That is, we can find $x'' \in \R^n$ such that $x'' \left( i \right) = 0$ for all $i \in W^C$,
    $\left| x'' \left( i \right) - x' \left( i \right) \right| \le \tfrac{\left( 1 - \eps \right) \delta}{\eps \sqrt{n}}$ for all $i \in W$, and
    $\left| \left \langle x'', \vec{1} \right \rangle \right| \leq \tfrac{2}{p} \tfrac{\log\left( n \right)}{\sqrt{n}}$.
    Now $x'' \left( i \right) / R$ might not be an integer for $i \in W$,
    but by changing each coordinate by at most $\delta / \sqrt{n}$, this can be achieved.
    Moreover, this can be done in such a way (by alternating the sign of the change in the coordinates)
    that the inner product of this vector with $\vec{1}$ changes by at most $\delta / n$.
    That is, we can find $x \in \R^n$ such that $x \left( i \right) = 0$ for all $i \in W^C$,
    $\left| x \left( i \right) - x'' \left( i \right) \right| \leq \tfrac{\delta}{\sqrt{n}}$ for all $i \in W$,
    $x \left( i \right) / R \in \Z$, and
    $\left| \left \langle x, \vec{1} \right \rangle - \left \langle x'' , \vec{1} \right \rangle \right| \leq \delta / n$.
    Consequently we must have
    $\left| \left \langle x, \vec{1} \right \rangle \right| \leq \tfrac{2}{p} \tfrac{\log\left( n \right)}{\sqrt{n}} + \delta / n \leq \tfrac{3}{p} \tfrac{\log\left( n \right)}{\sqrt{n}}$.
    By construction, we must have $\left\| x \right\| \in \left[ 1 - 2 \delta / \sqrt{\epsilon}, 1 + 2 \delta / \sqrt{\epsilon} \right]$.
    Finally, since $\left| x \left( i \right) - v \left( i \right) \right| \leq \tfrac{\delta}{\eps \sqrt{n}} \leq \tfrac{\log \left( n \right)}{n^{1/8}}$ for all $i \in W$,
    we have
    $\left\| x|_{S' \left( x \right)} \right\| \geq \left\| v|_{S \left( v \right)} \right\| - 2 \delta / \sqrt{\eps} \geq 1/10 - 4 \delta \geq 1/20$.
    Thus we have $x \in \Lambda_W$ and also $v - x \in \left[ - \tfrac{\delta}{\eps \sqrt{n}}, \tfrac{\delta}{\eps \sqrt{n}} \right]^n$.
    By the assumption $\epsilon > 1/4$ we get \eqref{eq:netprop2}.

    \medskip

	\textbf{Step 3:} Our next goal is to bound the cardinality of $\Lambda_W$. First note that $\Lambda_W$ is contained in the ball of radius $1 + 2 \delta$ in $\R^{\eps n}$. If we cover $\Lambda_W$ with hypercubes of edgelength $R$ (i.e., each point in $\Lambda_W$ is covered by at least one vertex of such a hypercube),
    then the union of these hypercubes will be contained in the ball of radius $2$ in $\R^{\eps n}$.
    Note that each such hypercube has $2^{\eps n}$ vertices.
    Recall that the Euclidean ball of radius $r$ in $\R^d$ has volume $V_d \left( r \right) = \tfrac{\pi^{d/2}}{\Gamma \left( d/2 + 1 \right)} r^d$,
    which by Stirling's approximation is at most $\left( 2 \pi e / d \right)^{d/2} r^d$.
    Consequently we have the following bound on the cardinality of $\Lambda_W$:
    \begin{equation} \label{eq:net-size2}
    \left| \Lambda_W \right|
    	\leq
    	2^{\eps n} \frac{V_{\eps n} \left( 2 \right) }{R^{\eps n}}
    	\leq
    	2^{\eps n} \frac{ \left( 8 \pi e / \left( \eps n \right) \right)^{\eps n / 2} }{\left( \delta / \sqrt{n} \right)^{\eps n}}
    	\leq
    	\left( \frac{40}{\delta} \right)^{\eps n}.
    \end{equation}
	We use this estimate later when we take a union bound over points in $\Lambda_W$.

    \medskip

	\textbf{Step 4:} For a vector $v \in \R^n$, define the event
    \[
	F_{W,v,\wh{A}}
	:= \left\{
	    Q_{W^C} P_{W^C} \wh{A} v \in B \left( 0, \delta  \log \left( n \right) / \sqrt{n} \right)
	\right\},
    \]
    and for a point $x \in \Lambda_W$ define
    \[
      H_{W, x, \wh{A}}
	:= \left\{
	    Q_{W^C} P_{W^C} \wh{A} x \in B \left( 0, \frac{12}{p} \frac{\delta \log\left( n \right)}{\sqrt{n}} \right)
	\right\},
    \]
    where again
    $P_{W^C}$ is the projection onto the coordinates of $W^C$, and $Q_{W^C}$ is the orthogonal projection onto the space orthogonal to $\vec{1}_{W^C}$.
    For $G \in \Typ$, we now prove that
    \begin{equation} \label{eq:impl1}
    v_2 \in \Omega_W \Rightarrow F_{W, v_2, \wh{A}} \mbox{ holds.}
    \end{equation}
    Moreover, we shall also see that
    \begin{equation} \label{eq:impl2}
    v_2 \in \Omega_W \mbox{ and } F_{W, v_2, \wh{A}} \mbox{ holds} \Rightarrow \exists x \in \Lambda_W \mbox { such that } H_{W, x, \wh{A}} \mbox{ holds.}
    \end{equation}

    Let us prove the implication \eqref{eq:impl1}. By definition,
    $\left( \wh{A} v_2 \right) \left( i \right) = \lambda_2 v_2 \left( i \right)$
    for all coordinates $i \in \left[ n \right]$.
    Since $G \in \Typ$, $\left| \lambda_2 \right| \leq 8 / \sqrt{np}$.
    Since $v_2 \in \Omega_W$, $\left| v_2 \left( i \right) \right| \leq \delta / \sqrt{n}$ for all $i \in W^C$,
    and so $\left| \left( P_{W^C} \wh{A} v_2 \right) \left( i \right) \right|
    \leq 8 \delta / (n \sqrt{p})$ for all $i \in W^C$.
    Therefore
    $P_{W^C} \wh{A} v_2 \in B \left( 0, 8 \delta / \sqrt{np} \right) \subseteq B \left( 0, \delta \log \left( n \right) / \sqrt{n} \right)$.
    Since $\left\| Q_{W^C} \right\| = 1$, it follows that the event $F_{W, v_2, \wh{A}}$ holds.

    Let us now prove \eqref{eq:impl2}. For $v \in \Omega_W$, let $x \in \Lambda_W$ be the closest point in $\Lambda_W$ such that
    $u := v - x \in \left[ - \tfrac{4 \delta}{\sqrt{n}}, \tfrac{4 \delta}{\sqrt{n}} \right]^n$;
    as we discussed above, such an $x \in \Lambda_W$ exists.
    Then
    \[
      \left\| Q_{W^C} P_{W^C} \wh{A} x \right\|
      \leq
      \left\| Q_{W^C} P_{W^C} \wh{A} v \right\|
      +
      \left\| Q_{W^C} P_{W^C} \wh{A} u \right\|
      \leq
      \left\| Q_{W^C} P_{W^C} \wh{A} v \right\|
      + \frac{8 \delta \log \left( n \right)}{p \sqrt{n}},
    \]
    where the first inequality follows from the triangle inequality,
    and the second inequality follows from the Cauchy-Schwarz inequality,
    the fact that $\left\| u \right\| \leq 4 \delta$, and \prettyref{thm:typical-props}, part \ref{pr:mat-proj}.
    Therefore, if $G \in \Typ$, $v_2 \in \Omega_W$, and the event $F_{W, v_2, \wh{A}}$ holds,
    then there exists $x \in \Lambda_W$ such that the event $H_{W, x, \wh{A}}$ holds.

    \medskip

	\textbf{Step 5:} Fix $x \in \Lambda_W$. Our next goal is to prove the bound
	\begin{equation} \label{eq:boundH}
	\PT \left[ H_{W,x,\wh{A}} \right] \leq 2 n^6 \left( \frac{C_2}{\sqrt{p(1-p)}} \delta \left( \log \left( n \right) \right)^{3/2}  \right)^{\left( 1 - \eps \right) n}
	\end{equation}
	for a universal constant $C_2 > 0$.

    We do this by coupling $P_{W^C} \wh{A} x$ with a vector
    whose nonzero entries are independent sums of scaled independent Bernoulli random variables,
    in order to shed the correlations introduced by the degrees in the normalization of $\wh{A}$.
    We can write $\wh{A} = \tfrac{1}{np} A + \tilde{A}$, where $\tilde{A}$ is a correction matrix.
    Then by the triangle inequality and \prettyref{thm:typical-props}, part \ref{pr:subtracting-mean}, we have
    \[
	\frac{1}{np} \left\| Q_{W^C} P_{W^C} A x \right\|
	\leq
	\left\| Q_{W^C} P_{W^C} \wh{A} x \right\|
	+ \left\| Q_{W^C} P_{W^C} \tilde{A} x \right\|
	\leq
	\left\| Q_{W^C} P_{W^C} \wh{A} x \right\|
	+ 4 p^{-7/2} \frac{\left( \log \left( n \right) \right)^2}{n},
    \]
    where in the application of \prettyref{thm:typical-props}, part \ref{pr:subtracting-mean} we used that
    $x \in \Lambda_W$ and so $\left| \left \langle x, \vec{1} \right \rangle \right| \leq \left( 3/p \right) \log \left( n \right) / \sqrt{n}$.
    So for $n$ large enough so that $p^{-5/2} \leq \frac{\delta \sqrt{n}}{\log n}$ holds, we have
    \begin{align*}
	\PT \left( H_{W, x, \wh{A}} \right)
	&\leq
	\PT \left( \frac{1}{np} Q_{W^C} P_{W^C} A x \in B \left( 0, \frac{16}{p} \frac{\delta \log \left( n \right)}{\sqrt{n}} \right) \right) \\
	&=
	\PT \left( Q_{W^C} P_{W^C} A x \in B \left( 0, 16 \delta \sqrt{n} \log \left( n \right) \right) \right) \\
	&\leq
	2 \P \left( Q_{W^C} P_{W^C} A x \in B \left( 0, 16 \delta \sqrt{n} \log \left( n \right) \right) \right),
    \end{align*}
    where in the last line we used the fact that $G \in \Typ$ with high probability.
    We are thus left with bounding this latter probability.

    Define $Y := P_{W^C} A x$.
    Note that $x_i = 0$ for all $i \in W^C$, so we can write $x = P_W x$, where $P_W$ is the coordinate projection onto $W$.
    We have $Y_i = 0$ for $i \in W$, while for $i \in W^C$ we have
    $Y_i = \sum_{j = 1}^n A_{ij} x_j = \sum_{j \in W} A_{ij} x_j$,
    that is, $Y_i$ is a sum of scaled independent Bernoulli random variables.
    In order to bound the small ball probability for $Y_i$ using the Littlewood-Offord-type estimate,
    we need to guarantee that $x$ has many entries with large enough magnitude.
    Contrary to the proof in \prettyref{sec:adj_union_bound}, we now do not have a bound on $\left\| x \right\|_{\infty}$,
    and so we cannot deduce this immediately.
    Instead, we use the fact that $\left\| x|_{S' \left( x \right)} \right\| \geq 1/20$.
    Recall that by definition this means that
    \[
      \sum_{i : \left| x_i \right| \leq \tfrac{2 \log \left( n \right)}{n^{1/8}} } \left| x_i \right|^2 \geq \left( 1/20 \right)^2,
    \]
    and since $\sum_{i : \left| x_i \right| \leq 1/n } \left| x_i \right|^2 \leq n \times \left( 1 / n^2 \right) = 1/n$, for large enough $n$ we have
    \[
      \sum_{i : \tfrac{1}{n} \leq \left| x_i \right| \leq \tfrac{2 \log \left( n \right)}{n^{1/8}} } \left| x_i \right|^2 \geq 10^{-3}.
    \]
    Define for each $\ell \in \Z$ the set of indices
    $S_{\ell} := \left\{ i \in \left[ n \right] : \left| x_i \right| \in \big[ 2^{- \left( \ell + 1 \right)}, 2^{- \ell} \big) \right\}$.
    Then by the above we have
    \[
     \sum_{\ell = \tfrac{1}{10} \log \left( n \right)}^{\log \left( n \right)} \left| S_{\ell} \right| 4^{-\ell} \geq 10^{-3},
    \]
    and so there must exist an integer $\ell^* \in \left[ \tfrac{1}{10} \log \left( n \right), \log \left( n \right) \right]$
    such that $\left| S_{\ell^*} \right| 4^{- \ell^*} \geq 10^{-3} / \log \left( n \right)$.
    We can now apply \prettyref{lem:interval_prob} to the random variable $2^{\ell^* + 1} Y_i$
    with $m = \tfrac{4^{\ell^*}}{10^3 \log \left( n \right)}$ and $r = 2^{\ell^*+1} \times 16 \delta \log \left( n \right)$
    to get that
    \begin{equation*}
	\conc \left( Y_i, 16 \delta \log \left( n \right) \right)
	= \conc \left( 2^{\ell^* + 1} Y_i, 2^{\ell^* + 1} 16 \delta \log \left( n \right) \right)
	\leq \frac{C_1}{\sqrt{p(1-p)}} \delta \left( \log \left( n \right) \right)^{3/2}
    \end{equation*}
    for some universal constant $C_1>0$.

    Furthermore, the random variables $\left\{ Y_i \right\}_{W^C}$ are independent,
    as we have already argued in \prettyref{sec:adj_union_bound}.
    Thus, we can apply \prettyref{lem:RV} to $Y$, with
    $t = 16 \delta \log \left( n \right)$,
    $d = \left( 1 - \eps \right) n$,
    $K=n^4$,
    $q = \frac{C_1}{\sqrt{p(1-p)}} \delta \left( \log \left( n \right) \right)^{3/2} $, and
    $P = Q_{W^C}$, to get
    \begin{align*}
      \P \left( Q_{W^C} P_{W^C} A x \in B \left( 0, 16 \delta \sqrt{n} \log \left( n \right) \right) \right)
      &\leq \conc \left( Q_{W^C} Y, 16 \delta \sqrt{n} \log \left( n \right) \right) \\
      &\leq \left( \frac{C_2}{\sqrt{p(1-p)}} \delta \left( \log \left( n \right) \right)^{3/2}  \right)^{\left( 1 - \eps \right) n} \times n^6
    \end{align*}
    for some universal constant $C_2>0$. Thus the bound \eqref{eq:boundH} is proven.

    \medskip

	\textbf{Step 6:}
    Finally, we take a union bound to arrive at our result:
    \begin{multline*}
      \PT \left( \left| v_2 \left( i \right) \right| \leq \frac{\delta}{\sqrt{n}} \ \text{ for all } \  i \in W^C \right) \\
	\begin{aligned}
	    & \stackrel{\eqref{eq:impl1} }{\leq} \PT \left( F_{W, v_2, \wh{A}} \right)
	    \stackrel{\eqref{eq:impl2} }{\leq} \PT \left( \cup_{x \in \Lambda_W} H_{W, x, \wh{A}} \right)
	    \leq \left| \Lambda_W \right| \max_{x \in \Lambda_W} \PT \left( H_{W, x, \wh{A}} \right) \\
	    & \stackrel{\eqref{eq:net-size2} \wedge \eqref{eq:boundH} }{\leq}  2 n^6 \left( \frac{40}{\delta} \right)^{\eps n} \times  \left( \frac{C_2}{\sqrt{p(1-p)}} \delta \left( \log \left( n \right) \right)^{3/2}  \right)^{\left( 1 - \eps \right) n} \\
	    &\leq  \left( C_p' \left( \log \left( n \right) \right)^{3/2} \right)^n  \delta^{\left( 1 - 2 \eps \right) n}
	\end{aligned}
    \end{multline*}
    for some constant $C_p'$ depending only on $p$.
\end{proof}

\medskip

Using this lemma we now prove \prettyref{thm:deloc}.

\medskip

\begin{proof}
    Fix a constant $\eta > 0$, and let $\eps = 1/2 - \eta$.
    We apply \prettyref{lem:fixed-W},
    and take a union bound over the possible subsets $W\subseteq\left[n\right]$
    (of which there are at most $2^n$).
    The lemma thus tells us that there exists a constant $C$
    such that,
    conditioned on $G \in \Typ$,
    the probability that there exists a subset $W \subseteq \left[n \right]$ of size $\eps n$
    such that $\left| v_2 \left( i \right) \right| \leq \delta / \sqrt{n}$ for all $i \in W^C$
    is at most
    \[
      \left( C \log \left( n \right) \right)^{2n} \times \delta^{2\eta n}.
    \]
    Now choosing $\delta := \left( \log \left( n \right) \right)^{-2/\eta}$, we get that
    conditioned on $G \in \Typ$,
    the probability that there are not at least $\left( 1/2 - \eta \right) n$ entries of each eigenvector of $\wh{A}$ of magnitude at least
    \[
      \frac{1}{\sqrt{n} \left( \log \left( n \right) \right)^{2/\eta}}
    \]
    is at most $\left( C / \log \left( n \right) \right)^{2n}$.
    Using Theorem \ref{thm:typical} we know that $G \in \Typ$ asymptotically almost surely, and we are done.
\end{proof}

%\section{Open problems}\label{sec:open}
%\input{files/open}

%%%%%%%%%%%%%%%%%%%%%%%%
%%% Acknowledgements %%%
%%%%%%%%%%%%%%%%%%%%%%%%

\section*{Acknowledgements}

We are grateful to Fan Chung for the question and for a useful discussion. 
This project was initiated at the Simons Institute for the Theory of Computing at UC Berkeley during the Algorithmic Spectral Graph Theory semester in Fall 2014, and we thank the Simons Institute for its hospitality. 
M.Z.R.\ gratefully acknowledges support from NSF grant DMS 1106999. 
T.S.\ gratefully acknowledges support from an NSF graduate research
fellowship, grant no.\ DGE 1106400.

%%%%%%%%%%%%%%%%%%
%%% References %%%
%%%%%%%%%%%%%%%%%%

\bibliographystyle{abbrv}

%%%%%%%%%%%%%%%%
%%% Appendix %%%
%%%%%%%%%%%%%%%%

\appendix
\renewcommand{\thesection}{\Alph{section}}

\section{A formula for the Dirichlet form}
\label{app:calc}
The following calculation to simplify the expression for $f^T \nL_{G_+} f$ 
in the proof of Lemma~\ref{lem:adding_edge} 
is straightforward but slightly cumbersome. We have:
\begin{align*}
	f^T \nL_{G_+} f
	&= \sum_{\substack{\left\{i,j\right\} \in E\\ \left\{i,j\right\} \cap \left\{ u,v \right\} = \emptyset}} \left(\frac{1}{\sqrt{d_i}}f(i) -\frac{1}{\sqrt{d_j}}f(j)\right)^2
	+ \sum_{\substack{j\sim u \\ j \neq v}} \left(\frac{1}{\sqrt{d_u+1}}f(u) -\frac{1}{\sqrt{d_j}}f(j)\right)^2\\
	&\quad + \sum_{\substack{i\sim v \\ i \neq u}} \left(\frac{1}{\sqrt{d_v+1}}f(v) -\frac{1}{\sqrt{d_i}}f(i)\right)^2 + \left(\frac{1}{\sqrt{d_u+1}}f(u) -\frac{1}{\sqrt{d_v + 1}}f(v)\right)^2\\
	&= f^T \nL_{G} f  \\
	&\quad + \sum_{\substack{j\sim u \\ j \neq v}} \left\{ \left(\frac{1}{\sqrt{d_u+1}}f(u) -\frac{1}{\sqrt{d_j}}f(j)\right)^2 - \left(\frac{1}{\sqrt{d_u}}f(u) -\frac{1}{\sqrt{d_j}}f(j)\right)^2 \right\} \\
 &\quad + \sum_{\substack{i\sim v\\i\neq u}} \left\{ \left(\frac{1}{\sqrt{d_v +1}} f(v)-\frac{1}{\sqrt{d_i}}f(i)\right)^2
	- \left(\frac{1}{\sqrt{d_v}} f(v)-\frac{1}{\sqrt{d_i}}f(i)\right)^2 \right\}
	 \\
	 &\quad + \left(\frac{1}{\sqrt{d_u + 1}}f(u) - \frac{1}{\sqrt{d_v + 1}}f(v)\right)^2\\
	 &= \lambda_2(\nL_G)
	  +  2\frac{\sqrt{d_u+1}-\sqrt{d_u}}{\sqrt{d_u + 1}} \sum_{\substack{j\sim u\\j\neq v}} \frac{1}{\sqrt{d_u d_j}}f(u)f(j) \\
	 &\quad + 2\frac{\sqrt{d_v+1} - \sqrt{d_v}}{\sqrt{d_v + 1}}\sum_{\substack{i\sim v\\i\neq u}}\frac{1}{\sqrt{d_i d_v}} f(v)f(i) - \frac{2f(u)f(v)}{\sqrt{(d_u+1)(d_v+1)}}.
\end{align*}
Since $f$ is the second eigenvector of $\nL_G$, we have for every vertex $i \in V$ that
\[
 \lambda_2 \left( \nL_G \right) f \left( i \right) = \left( \nL_G f \right) \left( i \right) = f \left( i \right) - \frac{1}{\sqrt{d_i}} \sum_{j \sim i} \frac{1}{\sqrt{d_j}} f\left( j \right),
\]
and so for every $i \in V$ we have 
\[
 \sum_{j \sim i} \frac{1}{\sqrt{d_i d_j}} f\left( i \right) f \left( j \right) = \left( 1 - \lambda_2 \left( \nL_G \right) \right) f\left( i \right)^2.
\]
Thus we arrive at the expression in~\eqref{eq:calc_final}.

\section{Additional properties of typical Erd\H{o}s-R\'enyi random graphs}
\label{app:typical}

In this appendix we prove \prettyref{thm:typical-props}, which states that several additional
spectral properties hold for typical instances of $G(n,p)$ (as defined in~\prettyref{def:typical}).
It is likely that these lemmas have been proven before and are well known; we include them here for
completeness.

We first prove that the top eigenvectors of $A$ and $\wh A$ are close to $\vec{1}$ for typical instances of $G(n,p)$.

\begin{lemma}\label{lem:normalized_top_ev}
    Let $n$ and $p$ be such that $p \in \left(\tfrac {10}{\sqrt{n}},1 \right)$. Let $G\in \cT_{n,p}$.
    Then the top eigenvectors $v_1(A)$ and $v_1(\wh A)$
    of the unnormalized and the symmetric normalized adjacency matrices
    are close to $\vec{1}$, in the following specific sense:
    \[
	\Vert v_1(A) - \vec{1} \Vert_2 \le 2 \frac{\log (n)}{\sqrt{p n}}, \qquad \text{and} \qquad
	\Vert v_1(\wh A) - \vec{1} \Vert_2 \le \frac{2}{p} \frac{\log
\left( n \right)}{\sqrt{n}},
    \]
    for all $n$ large enough.
\end{lemma}

\begin{proof}
    The top eigenvector $v_1 ( \wh A )$ of $\wh A$ is explicitly known, and so the corresponding statement is simple to prove.
    We know that for a node $i$, $v_1(\wh A) \left( i \right) =
    \sqrt{d_i} / \sqrt{\sum_{j} d_j}$, while $\vec{1} \left( i \right) =
    1/ \sqrt{n}$. Since $G \in \cT_{n,p}$, we have that for all $i$, $np - \log(n) \sqrt{n} \leq d_i \leq  np + \log(n) \sqrt{n}$, and that $\sum_j d_j \in
\left[ n^2 p - \log (n) n, n^2 p + \log (n) n \right]$.  Using these estimates,
we have that for all $i$,
\[
    \left| v_1(\wh A) \left( i \right) - \vec{1} \left( i \right) \right|
    \leq \frac{2}{p} \frac{\log \left( n \right) }{n}.
\]
This then directly implies that $\Vert v_1(\wh A) - \vec{1} \Vert_2 \leq
\frac{2}{p} \frac{\log \left( n \right) }{\sqrt{n}}$.

The top eigenvector $v_1 ( A )$ of the unnormalized adjacency matrix $A$ does not have an explicit formula. However, Mitra~\cite{mitra2009entrywise} proved nearly optimal entrywise bounds for $v_1 (A)$,
from which the desired result follows.
\end{proof}

The following lemma is an extension of the result above, showing that
projecting $A$ and $\wh A$ onto the space orthogonal to $\vec{1}$ yields a
matrix with smaller norm.

\begin{lemma} \label{lem:proj_norm}
    Let $n \in \mathbb{N}$. Let $S \subseteq \left[n \right]$, let $P_S$ denote the coordinate projection onto $S$, and let $Q_S$ denote the orthogonal projection onto the space orthogonal to $\vec{1}_S$.
    Let $p \in \left( \tfrac{10}{\sqrt{n}}, 1 \right)$, and let $G\in\cT_{n,p}$.
    Then we have
    \[
	\Vert Q_S P_{S} \wh A \Vert_2 \leq \frac{2}{p} \frac{\log \left( n \right)}{\sqrt{n}},
	\qquad \text{and} \qquad
	\Vert Q_S P_{S} A \Vert_2 \leq 2\sqrt{\frac{n}{p}} \log \left( n \right)
    \]
    for all $n$ large enough.
\end{lemma}

\begin{proof}
    The proofs for $A$ and $\wh A$ proceed identically; we give the
    proof only for $\wh A$.
    Recall that for a matrix $M$, $\left\| M \right\|_2 = \sup_{x: \left\|x \right\|_2 = 1} \left\| M x \right\|_2$.

    For any unit vector $x$ we can write $x = \alpha v_1 ( \wh A ) + \sqrt{1 - \alpha^2} y$ for some $\alpha \in \left[-1,1\right]$ and some unit vector $y$ that is orthogonal to $v_1 ( \wh A )$.
    By the triangle inequality we have that
    \[
	\left\| Q_S P_S \wh A x \right\|_2 \leq \left\| Q_S P_S \wh A v_1 ( \wh A ) \right\|_2 + \left\| Q_S P_S \wh A y \right\|_2,
    \]
    and we bound each term separately.
    For the second term, we can use the submultiplicativity of the norm and the bounds on the eigenvalues of $G \in \cT_{n,p}$:
    \[
	\left\| Q_S P_S \wh A y \right\| \leq \left\| Q_S \right\| \left\| P_S \right\| \left\| \wh A y \right\| \leq \left\| \wh A y \right\| \leq \max_{i \geq 2} \left| \lambda_i ( \wh A ) \right| \leq  \frac{8}{\sqrt{pn}}.
    \]
    For the first term, note that $\wh A v_1 (\wh A) = v_1 (\wh A)$, and we then have that
    \begin{align*}
	\left\| Q_S P_S v_1 (\wh A) \right\| &\leq \left\| Q_S P_S \vec{1} \right\| + \left\| Q_S P_S \left( v_1 (\wh A) - \vec{1} \right) \right\| \\
	&\leq 0 + \left\| Q_S \right\| \left\| P_S \right \| \left\| v_1 ( \wh A ) - \vec{1} \right\| \leq \frac{2}{p} \frac{\log \left( n \right)}{\sqrt{n}},
    \end{align*}
    where the last inequality follows from~\prettyref{lem:normalized_top_ev}.
\end{proof}

The next lemma shows that for any subset of coordinates $S \subseteq \left[n\right]$,
the matrices
$\tfrac{1}{np}Q_SP_{S}A$ and $Q_SP_{S}\wh A$ behave similarly.

\begin{lemma}\label{lem:normalization}
    Let $n \geq 10$. Let $S \subseteq \left[n \right]$, let $P_S$ denote the coordinate projection onto $S$, and let $Q_S$ denote the orthogonal projection onto the space orthogonal to $\vec{1}_S$.
    Let $p \in \left( \tfrac{\log n}{\sqrt{n}}, 1 \right)$ and let $G \in \cT_{n,p}$.
    Then for any unit vector $x \in \R^n$ with $|\langle x, \vec{1} \rangle| \le \alpha$, we have that
    \[
	\Vert Q_S P_{S} \wh A x - \tfrac{1}{np} Q_S P_{S} A x\Vert
	\le 6 p^{-5/2} \frac{\left( \log \left( n \right) \right)^2 + \alpha \sqrt{n} \log \left( n \right)}{n}.
    \]

\end{lemma}

\begin{proof}
    We can write
    \[
	\wh A = \left( \tfrac{1}{\sqrt{np}} I + \tilde D \right) A \left( \tfrac{1}{\sqrt{np}}I + \tilde D \right),
    \]
    where $\tilde D$ is a diagonal correction matrix with entries $\tilde D_{ii} = 1/\sqrt{d_i} - 1/\sqrt{np}$. Since $G \in \cT_{n,p}$, $np - \sqrt{n} \log \left( n \right) \leq d_i \leq np + \sqrt{n} \log \left( n \right)$ for all $i$, and thus it follows that $\left| \tilde D_{ii} \right| \leq \frac{1}{p^{3/2}} \frac{\log \left( n \right)}{n}$ for all $i$.
    By the
    triangle inequality we thus have that
    \[
	\Vert Q_S P_{S} \wh A x - \tfrac{1}{np} Q_S P_{S}A x \Vert
	\le
	\Vert Q_S P_{S}\tilde D A\tilde D x \Vert
	+ \tfrac{1}{\sqrt{np}}\Vert  Q_S P_{S}\tilde D A x \Vert
	+\tfrac{1}{\sqrt{np}}\Vert Q_S P_{S}A \tilde Dx \Vert,
    \]
    and we bound each term separately.

    The first term can be bounded simply by the submultiplicativity of the
    norm:
    \begin{align*}
	\Vert Q_S P_S \tilde D A \tilde D x \Vert
	&\le
	\Vert Q_S \Vert  \Vert P_{S}\Vert \Vert \tilde D \Vert
	\Vert A \Vert \Vert \tilde D \Vert \Vert x \Vert\\
	&\le 1 \times 1 \times \left( \frac{1}{p^{3/2}} \frac{\log \left( n \right)}{n} \right) \times \left( np + \sqrt{n} \log \left( n \right) \right) \times \left( \frac{1}{p^{3/2}} \frac{\log \left( n \right)}{n} \right) \times 1 \\
	&\le \frac{2}{p^2} \frac{\left( \log \left( n \right) \right)^2}{n},
    \end{align*}
    where the last inequality uses the fact that $p>\tfrac{\log n}{\sqrt n}$.

    Next we bound the third term, which can also be done by the submultiplicativity of the norm, together with \prettyref{lem:proj_norm}:
    \begin{align*}
	\tfrac{1}{\sqrt{np}} \Vert  Q_S P_{S} A  \tilde D x \Vert
	&\le \frac{1}{\sqrt{np}} \Vert Q_S P_{S} A  \Vert  \Vert \tilde D \Vert
	\Vert x \Vert \\
	&\le \frac{1}{\sqrt{np}} \left(\sqrt{\tfrac n p}\log \left( n \right) \right) \times \left( \frac{1}{p^{3/2}} \frac{\log \left( n \right)}{n} \right) \times 1 = p^{-5/2} \frac{ \left( \log \left( n \right) \right)^2}{n}.
    \end{align*}

    Finally, we bound the second term. By submultiplicativity again, we have that
    \[
     \tfrac{1}{\sqrt{np}} \Vert Q_S P_S \tilde{D} A x \Vert \leq \tfrac{1}{\sqrt{np}} \Vert Q_S \Vert \Vert P_S \Vert \Vert \tilde{D} \Vert \Vert A x \Vert \leq \frac{1}{p^2} \frac{\log \left( n \right)}{n^{3/2}} \Vert A x \Vert,
    \]
    so what remains is to bound $\Vert A x \Vert$. Let $\gamma := \left \langle x, \vec{1} \right \rangle$; by assumption $\left| \gamma \right| \leq \alpha$. We can then write $x = \gamma \vec{1} + \sqrt{1 - \gamma^2} z$ for some unit vector $z$ that is orthogonal to $\vec{1}$. Then we have that
    \[
     \Vert Ax \Vert = \left| \gamma \right| \Vert A \vec{1} \Vert + \sqrt{1 - \gamma^2} \Vert A z \Vert \leq \alpha \left( np + \sqrt{n} \log \left( n \right) \right) + \Vert A z \Vert \leq 2 \alpha np + \Vert A z \Vert,
    \]
    where the last inequality uses the fact that $p>\tfrac{\log n}{\sqrt n}$. So what remains is to bound $\Vert A z \Vert$.
    Let $\theta := \left\langle z, v_1 \left( A \right) \right \rangle$.
    By \prettyref{lem:normalized_top_ev} and the Cauchy-Schwarz inequality we have that
    \begin{align*}
     \left| \theta\right| =
     \left| \left \langle z, v_1 \left( A \right) \right \rangle \right|
     &= \left| \left \langle z, \vec{1} \right \rangle + \left \langle z, v_1 \left( A \right) - \vec{1} \right \rangle \right|
     = \left| 0 + \left \langle z, v_1 \left( A \right) - \vec{1} \right \rangle \right| \\
     &\leq \left\| z \right\| \left\| v_1 \left( A \right) - \vec{1} \right\| \leq 2 \frac{\log \left( n \right)}{\sqrt{pn}}.
    \end{align*}
    We can write $z = \theta v_1 \left( A \right) + \sqrt{1 - \theta^2} y$ for some unit vector $y$ that is orthogonal to $v_1 \left( A \right)$. Then using the triangle inequality we have that
    \[
      \left\| A z \right\| \leq \left| \theta \right| \left\| A \right\| + \left\| A y \right\| \leq \frac{2 \log \left( n \right)}{\sqrt{pn}} \left( n p + \sqrt{n} \log \left( n \right) \right) + 3 \sqrt{n p \left( 1 - p \right)} \leq 4 \sqrt{p n} \log \left( n \right),
    \]
    where the last inequality uses the fact that $p>\tfrac{\log n}{\sqrt n}$. Putting the previous displays together, we get that
    \[
     \tfrac{1}{\sqrt{np}} \Vert Q_S P_S \tilde{D} A x \Vert \leq \frac{2}{p} \frac{\log \left( n \right)}{n} \left( \alpha \sqrt{n} + \tfrac{2}{\sqrt p} \log \left( n \right) \right).
    \]
    The bounds on the three terms put together concludes the proof.
\end{proof}

\begin{lemma}\label{lem:ev2-bound}
    Fix $p \in (0,1)$ and let $n \geq 10$. Let $G \in\cT_{n,p}$. Then
    \[
	\lambda_2(\wh A) \ge \left( 1 - o \left( 1 \right) \right) \frac{1-p}{16\sqrt{np}}.
    \]
\end{lemma}
\begin{proof}
    For convenience, write $\lambda_i \equiv \lambda_i(\wh A)$.
    We know that $\sum_{i=1}^n \lambda_i = \Tr ( \wh{A} ) = 0$,
    because the diagonal entries of $\wh{A}$ are all zero.
   Since $\lambda_1(\wh A) = 1$, we then have
   \[
	\sum_{i> 1} \lambda_i = -1.
   \]
   If $G \in \cT_{n,p}$, then we have
   \[
       \sum_{i=1}^n \lambda_i^2
	= \Tr ( \wh{A}^2 )
       = \sum_{i,j \in [n]} \wh{A}_{ij}^2
       = 2 \sum_{\{u,v\} \in E(G)} \frac{1}{d_u d_v}
       \ge \frac{n^2 p - n \log \left( n \right) }{(np + \log \left( n \right) \cdot \sqrt{np})^2}
       = \frac{1}{p} \cdot (1 - o(1)).
   \]
   On the other hand, we have
    \[
     \sum_{i=2}^n \lambda_i^2
     \leq \max_{i > 1} \left| \lambda_i \right| \times \sum_{i=2}^n \left| \lambda_i \right|
     \leq \frac{8}{\sqrt{np}} \sum_{i=2}^n \left| \lambda_i \right|,
    \]
    where in the second inequality we used that $G \in \cT_{n,p}$. Putting together the two previous displays we get that
    \[
     \sum_{i=2}^n \left| \lambda_i \right|
     \geq \left( 1 - o \left( 1 \right) \right)
     \frac{\sqrt{n} \left( 1 - p \right)}{8 \sqrt{p}}.
    \]
    Let $k$ be such that $\lambda_2, \dots, \lambda_k \geq 0$ and $\lambda_{k+1}, \dots, \lambda_n < 0$. We then have
    \[
     2n \lambda_2 \geq 2 \sum_{i=2}^k \lambda_i = \sum_{i=2}^n \left| \lambda_i \right| + \sum_{i=2}^n \lambda_i
      \geq \left( 1 - o \left( 1 \right) \right)
     \frac{\sqrt{n} \left( 1 - p \right)}{8 \sqrt{p}}
      - 1
      = \left( 1 - o \left( 1 \right) \right)
     \frac{\sqrt{n} \left( 1 - p \right)}{8 \sqrt{p}},
    \]
    and dividing by $2n$ gives the claim.
\end{proof}

\begin{lemma} \label{lem:inf-norm}
    Fix $n\geq 10$ and $p \in (0,1)$. Let $G \in \cT_{n,p}$.
    Let $v$ be a unit eigenvector of $\wh A$ with eigenvalue $\lambda$, orthogonal to $v_1(\wh A)$.
    Fix $\alpha \in [0,1]$,
    and let $S = \{i \in [n] :  \vert v \left( i \right) \vert \le \alpha \}$.
    Then
    $\Vert P_S v \Vert_2 \ge \frac{1}{3} \left( \frac{|\lambda|}{\lambda_2} - \frac{\log
    (n)}{\alpha^4 \lambda_2 np} \right)$.

    In particular, if $\Theta(\lambda_2) = \lambda = \Theta(n^{-1/2})$ and
    $\alpha \ge \frac{\log (n)}{(np)^{1/8}}$, then
	$\Vert P_S v \Vert = \Theta(1)$.

    The analagous statement holds true for any eigenvector $u$ of $A$ with
    eigenvalue $O(n^{1/2})$.
\end{lemma}
\begin{proof}
    Let
    $B = \wh A - v_1v_1^T$, so that
    $\Vert B \Vert \le \lambda_2 ( \wh A )$.
    Let $v = v_S + v_T$, where $T = S^C$.
    By definition,
    \[
	|\lambda|
	= |\langle v, Bv \rangle|
	\le |\langle v_S, Bv_S\rangle| + 2|\langle v_S, B v_T \rangle|
	+ |\langle v_T,Bv_T\rangle|.
    \]
    The first two terms can be bounded above by $3\Vert v_S \Vert \Vert B
    \Vert \le 3\lambda_2 \Vert v_S \Vert$, using Cauchy-Schwarz.

    For the final term, we note that $\vert T \vert \le \frac{1}{\alpha^2}$,
    and since for $G \in \cT_{n,p}$, the entries of $B$ are bounded above in magnitude by
    $\frac{\log (n)}{np}$, we have that $\langle v_T, B v_T \rangle \le \vert T \vert^2
    \frac{\log (n)}{np} \le \frac{\log (n) }{\alpha^4 n p}$.
    Putting the bounds together, we have
    \[
	|\lambda| \le 3\lambda_2 \Vert v_S \Vert + \frac{\log (n)}{\alpha^4 np},
    \]
    and straightforward manipulation yields the result.
\end{proof}

\end{document}